\documentclass[a4paper,11pt]{amsart}
\usepackage[ngerman,english]{babel}
\usepackage{amssymb}
\usepackage{amsmath}
\usepackage{amsfonts}
\usepackage{amstext}
\usepackage{amsthm}
\usepackage{todonotes}
\usepackage{graphics}
\usepackage{graphicx}
\usepackage[T1]{fontenc}
\usepackage{textcomp}
\usepackage[font=footnotesize]{caption}
\usepackage[latin1]{inputenc}
\usepackage{enumerate}
\usepackage{hyperref}

\newcommand{\IN}{{\mathbb{N}}}
\newcommand{\IR}{{\mathbb{R}}}
\newcommand{\R}{{\mathbb{R}}}

\newcommand{\lm}{\lambda}

\newcommand{\1}{\mathbbmss{1}}

\renewcommand{\phi}{\varphi}
\renewcommand{\epsilon}{\varepsilon}

\newcommand{\aeq}{\Leftrightarrow}
\newcommand{\N}{\mathbb{N}}
\newcommand{\B}{\mathcal{B}}

\newcommand{\ltwo}{{L^2(X,m)}}

\newtheorem{theorem}{Theorem}[section]
\newtheorem{lemma}[theorem]{Lemma}
\newtheorem{proposition}[theorem]{Proposition}
\newtheorem{corollary}[theorem]{Corollary}

\theoremstyle{definition}
\newtheorem{definition}[theorem]{Definition}

\newtheorem{remarks}[theorem]{Remark}
\newtheorem{example}[theorem]{Example}

\newcommand{\Hmm}[1]{\leavevmode{\marginpar{\tiny%
			$\hbox to 0mm{\hspace*{-0.5mm}$\leftarrow$\hss}%
			\vcenter{\vrule depth 0.1mm height 0.1mm width \the\marginparwidth}%
			\hbox to 0mm{\hss$\rightarrow$\hspace*{-0.5mm}}$\\\relax\raggedright #1}}}

\begin{document}
\title[Courant's Theorem for Positivity Preserving Forms]
{Courant's Nodal Domain Theorem for Positivity Preserving Forms}
\author[M.~Keller]{Matthias Keller}
\address{M.~Keller,   Institut f\"ur Mathematik, Universit\"at Potsdam
	\\14476  Potsdam, Germany}
\email{{matthias.keller@uni-potsdam.de}}
\author[M.~Schwarz]{Michael Schwarz}
\address{M.~Schwarz,  Institut f\"ur Mathematik, Universit\"at Potsdam
	\\14476  Potsdam, Germany}
\email{mschwarz@math.uni-potsdam.de}

\begin{abstract}
	We introduce a notion of nodal domains for positivity preserving forms. This notion generalizes the classical ones for Laplacians on domains and on graphs. We prove the Courant nodal domain theorem in this generalized setting using  purely analytical methods. 
\end{abstract}

\maketitle
%\makeatletter
%\setlength\@fptop{0\p@}% float at the top
%\makeatother
%\tableofcontents

\section{Introduction}
In 1923 Courant, \cite{Courant}, proved his famous result about the nodal domains of the eigenfunctions. He considers eigenfunctions of a self-adjoint differential operator related to a quadratic form   $ Q_{\Omega}  $ on a bounded domain $ \Omega $ of $ \R^{d} $ with Dirichlet boundary conditions. A nodal domain is a maximal connected set where an eigenfunction does not change its sign. Ordering the eigenvalues $ \lm_{n} $ of the operator in increasing order counted with multiplicity he proved that
\begin{align*}
N_{Q_{\Omega}}(\lm_{n})\leq n,
\end{align*}
where $ N_{Q_{\Omega}}(\lm_{n}) $ is the number of nodal domains to the eigenvalue $ \lm_{n} $.

Since then this type of result was proven in various different contexts, like compact quantum graphs, \cite{GSW04}, and  finite discrete graphs, \cite{Davies}, and
 Schrödinger operators \cite{Alessandrini,Ancona}. In more specific settings there is indeed much more known about the geometry of the nodal domains, 
 see e.g. \cite{P56, C76, Bru78,DF88,NPS05, HH07, B08, M08, HHT09, P09,CM11, BBS12, BGRS12,XuYau12, B14, B15, BH16} and references therein.

The aim of this paper is to prove a Courant bound for positivity preserving  forms $ Q $ with compact resolvent, i.e., quadratic forms on an $ L^{2}$-space that  give rise to  self-adjoint compact resolvents $ (L+1)^{-1} $ that map positive functions to positive functions. Indeed, compact resolvent is equivalent to the fact that the self-adjoint operator $ L $ has purely discrete spectrum.
This setting includes numerous of the results mentioned above, e.g. elliptic second order differential operators and Schrödinger operators  on domains, manifolds, discrete (infinite) graphs and quantum graphs as well as operators coming from Dirichlet forms as they arise on fractals  which  are all examples of Schrödinger forms as treated in \cite{StVoi96}.

However, in this general setting there is a very fundamental question, namely: What is  a nodal domain? This can already be seen in the contrast of the definition of nodal domains for Laplacians  in the continuum and on discrete graphs. While in the continuum nodal domains are defined as topological connected components, in the discrete setting they are defined combinatorially. The two common features are that both definitions capture a notion of maximality and connectedness. In our setting we employ the notion of
\begin{itemize}
	\item  \emph{invariance} (Definition~\ref{definition:invariance}) which represents maximality and
	\item  \emph{irreducibility} (Definition~\ref{definition:irreducible}) which represents connectedness. 
\end{itemize}

For the special cases mentioned above this yields exactly the correct notions.

Both of these notions are not derived from the underlying space $ X $ but rather directly from the quadratic form $ Q $. To this end it is  necessary to restrict $ Q $ to a subspace of $ L^{2} $-functions supported on a subset of $ X $. In this paper we discuss two natural restrictions -- one for the general case and one for the special case of regular Dirichlet forms. In a very vague sense these restrictions can be thought of as  Dirichlet and Neumann boundary conditions. (Indeed, this relation can be made precise in the case of graphs, Section~\ref{section:graphs}.)

The first natural restriction of the form $ Q $  is given by $$ D(Q_{A})=\{f\in D(Q)\mid f\vert_{X\setminus A}=0\}   $$
for a measurable subset $ A\subseteq X$.
This yields the notion of \emph{$ Q $-nodal domains} which are introduced at the end of Section~\ref{section:nodal_domains}.
For the $ n $-th eigenvalue $ \lambda_{n} $, of the operator associated to $ Q $, with multiplicity $ k $, we prove
\begin{align*}
N_{Q}(\lambda_{n})\leq n+k-1
\end{align*}
in Corollary~\ref{corollary:courant_Q_nodal}.
Indeed, it can already be seen from the case of finite graphs \cite{Davies}, that the additional summand $ k-1 $ is necessary.

The second natural restriction which is relevant for the case  of a regular Dirichlet form $ Q $  is given by $$ D(Q^{A})= \overline{C_{c}(A)\cap D(Q)} $$
for an open subset $ A $ where the closure is take with respect to the form norm. This yields, c.f. Section~\ref{section:regular}, the notion of \emph{Dirichlet nodal domains} for $ Q $ whenever the eigenfunctions are continuous. In Corollary~\ref{corollary:reg:NodalDom}, we prove the same estimate as above for Dirichlet nodal domains. In the case of local Dirichlet forms which allow for a unique continuation principle the estimate holds even true without the term $ k-1 $, Corollary~\ref{corollary:strong_bound}.

To treat both cases simultaneously we introduce the notion of restrictions with respect to a nest. The main theorem, Theorem~\ref{NodalDom}, shows the Courant bound in this abstract setting. From this theorem Corollary~\ref{corollary:courant_Q_nodal} and Corollary~\ref{corollary:reg:NodalDom} can be deduced directly.
The major challenge in the proof is to show the existence and finiteness of the number of nodal domains in our setting. Here, we employ a result which is strongly inspired by a theorem of Gerlach, \cite{Gerlach}. Once one has established the notion of a nodal and has proven their existence and  finiteness, the proof of the Courant bound follows by  standard arguments.

To give an idea of the wide range of applicability of our results  we discuss a selection of specific examples. We start with the case of regular Dirichlet forms, Section~\ref{section:regular}, and become more specific by considering graphs, Section~\ref{section:graphs} and local Dirichlet forms in Section~\ref{section: local}. In Section~\ref{section: local} we consider the general case, Section~\ref{section:general_local}, elliptic differential operators, Section~\ref{section:elliptic} and forms with a doubly Feller resolvent, Section~\ref{section:Feller}.

On the technical level we need various operator theoretical facts for positivity preserving forms in the wide sense. In a somewhat more restricted settings these facts can be found in standard textbooks such as \cite{Fuku,reed_simon_IV} and they are  certainly well known to experts. However, we were not able to give a proper reference of these facts as they are applied in the setting of this paper. Therefore, for the reader's convenience  we included them in a series of appendices at the end of the paper.

\section{Irreducibility and connected components}\label{subsection:connected_components}
For the whole paper, let $X$ be a topological space and let $m$ be a measure on  $\B(X)$ of full support such that $L^2(X,m)$ is separable.   
Except for Section~\ref{section:regular}, these are the only assumptions on the underlying space.

For a measurable $A\subseteq X$, we will consider $L^2(A,m)$ as a subspace of $\ltwo$. We will not distinguish between an operator $T:L^2(A,m)\to L^2(A,m)$ and 
its extension to $\ltwo$ by $$Tf:=T(f1_A),\qquad f\in\ltwo.$$

In this section we introduce the notions of irreducibility and $Q$-connected components.  For the detailed definitions and properties  we refer to the appendices. Here, we summarize only the most important facts.

A \emph{form in the wide sense} $ Q $ is a form whose domain $ D(Q) $ is not necessarily dense in  $ L^{2}(X,m) $, see Appendix~\ref{section:forms}.  The resolvent $ G_{\alpha} $, $ \alpha>0 $,  of a closed form in the wide sense $ Q $ on $ L^{2}(X,m) $ is characterized  via the equality $ Q(G_\alpha f, g)+\alpha \langle G_\alpha f,g\rangle_\ltwo =\langle f,g\rangle_\ltwo$, $ f\in L^{2}(X,m),g\in D(Q)  $, Lemma~\ref{FormReso}. In the case when   $ Q $ is  \emph{positivity preserving}, i.e., $ Q $ is closed and $ Q(|f|)\leq Q(f) $ for all $ f\in D(Q) $, the resolvents map non-negative   functions to non-negative functions,  Lemma~\ref{ResPosPres}. In the following $ Q $ will always denote a positivity preserving form in the wide sense.

The concept of $Q$-connected components will be used in the next section to define nodal domains. The results in this section show that there exist nodal domains, and 
that the number of nodal domains is finite. To this end, we need the notion of $Q$-invariant sets, discussed in Appendix~\ref{section:invariance}. A measurable set $ A \subseteq X$ is called $ Q $-invariant if the multiplication operator of the characteristic function of $ A $ commutes with the resolvents, Definition~\ref{definition:invariance}. In this case for every $f\in D( Q) $  we have $ 1_{A}f\in D(Q) $ and $ Q(f)=Q(1_{A}f) +Q(1_{X\setminus A}f) $, Lemma~\ref{CharInv}.

\begin{definition}\label{definition:irreducible}
A positivity preserving form in the wide sense $Q$ is called \emph{irreducible} if every $Q$-invariant set $A$ satisfies $m(A)=0$ or $m(X\setminus A)=0$. 

We call a $Q$-invariant set $A$ with $m(A)>0$ a \emph{$Q$-connected component} if  $$ Q_A=Q|_{D(Q_A)}\quad\mbox{ with }D(Q_A)=\{u\in D(Q)\colon u=0 \text{ a.e. on } X\setminus A\} $$ 
is irreducible.     
\end{definition}
\begin{remarks}
 Irreducibility of a form in the wide sense can in some sense be seen as connectedness of the underlying space viewed through the eyes of the form. 
 Invariance of a set, on the other hand, means that for the form there is no interaction between the set and the complement of the set. In this sense, $Q$-connected components can be seen as connected components.
\end{remarks}

\begin{lemma}\label{IrredTrans}
 Let $Q$ be a positivity preserving form in the wide sense and let $A\subseteq X$ be invariant and of positive measure. Then, every  $Q_A$-connected component $B\subseteq A$ is a $Q$-connected component.
\end{lemma}

\begin{proof} 
By Lemma \ref{InvTrans} and the definition of $Q_A$-connected components the set $B$ is $Q$-invariant. By assumption the restriction $(Q_A)_B$ of $Q_A$ to $B$ is irreducible.
But we have $(Q_A)_B=Q_B$ and, hence, $Q_B$ is irreducible as well.
This concludes the proof.
\end{proof}

The following proposition shows that given a positivity preserving form in the wide sense with compact resolvent one can decompose the space $X$ in finitely many $Q$-connected components. 
The proof is similar to the proof of Theorem 5.1.6 in \cite{Gerlach}, where such a result is shown for non-symmetric positivity preserving semigroups on Banach lattices.
 This result will be used to decompose the spaces $\{f>0\}$ and $\{f<0\}$ for an eigenfunction $f$ of $Q$ into finitely many pieces. These pieces will be the nodal domains and the proposition
 shows that there are finitely many of them.  
  
\begin{proposition}[Gerlach's lemma]\label{Decomp}
Let $Q$ be a positivity preserving form in the wide sense with  compact resolvent $ G_{\alpha} $, $ \alpha>0 $. Let $f\in\ltwo$, $f>0$ be such that $G_1f\geq cf$ for a constant $c>0$.  Then, there 
are disjoint $Q$-connected components $C_1,\ldots,C_{l}$  such that $X=\bigcup_{i=1}^l C_i$. This decomposition is unique up to sets of measure zero. 
\end{proposition}
\begin{proof}

%By Lemma \ref{CharInv} and Propsition \ref{ResInv} one only has to find such a decomposition for one operator in the resolvent. Let $G_\alpha$ be such an operator and set $T=\alpha G_\alpha$. Of course it is equivalent to construct the decomposition for $T$. 
First assume that there is no $Q$-connected component. 
Inductively, we construct a nested sequence of $Q$-invariant sets $A_n$, $ n\geq 0 $, such that $\|f1_{A_n}\|_\ltwo\leq\frac{1}{2^n}\|f\|_\ltwo$. This is done the following way: Let $ A_{0}=X $. Now, suppose we have already constructed a $Q$-invariant set $A_n$. By assumption, $A_n$ is not a $Q$-connected component. Thus, there are disjoint
$Q_{A_n}$-invariant sets $B,B'$ of positive measure with $A_n=B\cup B'$. Without loss of generality,  $ B $ satisfies \[\|f1_{B}\|_\ltwo\leq \frac12 \|f1_{A_n}\|_\ltwo\leq\frac{1}{2^{n+1}}\|f\|_\ltwo\] 
while the second inequality follows then by the induction hypothesis.
By the first part of Lemma~\ref{InvTrans} the $Q_{A_n}$-invariant set $B$ is $Q$-invariant as well. We set $A_{n+1}=B$. Then,  the sequence defined by \[g_k=\frac{f1_{A_k}}{\|f1_{A_k}\|_\ltwo},\qquad k\in\N\] is bounded in $\ltwo$ and, by compactness, $(G_1 g_k)$ is (without loss of generality) $\ltwo$-convergent with limit $g$. We deduce \[\|g\|=\lim\limits_{k\to\infty}\frac{\|G_1(1_{A_k}f)\|_\ltwo}{\|f1_{A_k}\|_\ltwo}=\lim\limits_{k\to\infty}\frac{\|1_{A_k}G_1f\|_\ltwo}{\|f1_{A_k}\|_\ltwo}\geq c>0,\] where we used the $Q$- (and, thus, $G_1$-) invariance of $A_k$ and $G_1f\geq cf>0$.

On the other hand, \[g=g1_{A_k}\] holds for every $k$, since $A_k$ is a nested sequence and every $A_k$ is invariant with respect to $G_1$. Furthermore, by $\|f1_{A_k}\|_\ltwo\leq\frac{1}{2^k}\|f\|_\ltwo$ we infer 
\[f=\lim\limits_{k\to\infty} f1_{X\setminus A_k}\quad\text{in }L^2(X,m).\] This yields \[\langle f,g\rangle_\ltwo=\lim\limits_{k\to\infty}\langle f1_{X\setminus A_k},g\rangle_\ltwo=\lim\limits_{k\to\infty}\langle f1_{X\setminus A_k} ,g1_{A_k}\rangle=0.\]
But, since $f>0$ and $g\geq 0$ almost everywhere, we get $\|g\|_\ltwo=0$, which is a contradiction to $\|g\|_\ltwo\geq c$. 

Hence, we conclude that there is a $Q$-connected component $C_1\subseteq A$.
 Inductively, we construct a family of disjoint $Q$-connected components $C_i$ the following way: 
Suppose we have  disjoint $Q$-connected components $C_1,\ldots, C_n$. If $m(X\setminus \bigcup_{i=1}^n C_i)=0$, then we stop. Otherwise, we consider $Y=X\setminus \bigcup_{i=1}^n C_i$. Then, $Y$ is a $Q$-invariant 
set since the $Q$-invariant sets are a $\sigma$-algebra. Hence, by Proposition~\ref{RestrCpt} and Lemma~\ref{RestrRes} the form $Q_Y$ is a positivity preserving form 
in the wide sense with compact resolvent given by $G_\alpha^Y=G_\alpha|_{L^2(Y,m)},\alpha>0$. Furthermore, by $Q$-invariance of $Y$, we have \[G_1^Y (1_Y f)=G_1(1_Yf)=1_YG_1 f\geq c1_Yf.\] 
Hence, we can apply what we have proven for $Q$ above to $Q_Y$ as well and we infer the existence of a $Q_Y$-connected component $C_{n+1}$. 
By Lemma~\ref{IrredTrans} the $Q_Y$-connected component $C_{n+1}$ is a $Q$-connected component as well.

Suppose the family $(C_i)$ is infinite. Then, the  sequence defined by \[g_k=\frac{f1_{C_k}}{\|f1_{C_k}\|_\ltwo}, \qquad k\in\N,\] is bounded in $\ltwo$ and, by compactness, $(G_1 g_k)$ is (without loss of generality) $\ltwo$-convergent to a limit $g$. We deduce \[\|g\|=\lim\limits_{k\to\infty}\frac{\|G_1(1_{C_k}f)\|_\ltwo}{\|f1_{C_k}\|_\ltwo}=\lim\limits_{k\to\infty}\frac{\|1_{C_k}G_1f\|_\ltwo}{\|f1_{C_k}\|_\ltwo}\geq c>0,\] where we used the $Q$- (and, thus, $G_1$-) invariance of $C_k$ and $G_1f\geq cf>0$.

On the other hand, \[\langle g, G_1 g_k\rangle_\ltwo=0\] holds for every $k$, since the $C_k$ are disjoint and $Q$-invariant.  
This yields \[\langle g,g\rangle_\ltwo=\lim_{k\to\infty}\langle g,G_1 g_k\rangle_\ltwo=0.\]
Hence, we get $\|g\|_\ltwo=0$, which is a contradiction to $\|g\|_\ltwo\geq c$.  
Thus, the sequence $ (C_{n}) $ must be finite and we infer \[X=\bigcup_{i=1}^l C_i\] up to a null-set.

Finally, we show the uniqueness of the decomposition. Let $C_1,\ldots, C_l$ and $D_1,\ldots, D_k$ be two decompositions of $X$ into $Q$-connected components. Let $i\in\{1,\ldots,l\}$ 
be arbitrary. Then, we have  $C_i=\bigcup_{j=1}^k C_i\cap D_j $ up to a set of measure zero. By Lemma \ref{InvTrans} the sets  $C_i\cap D_j$ are $Q_{C_i}$-invariant. But since $Q_{C_i}$ is irreducible,
we infer $C_i\cap D_j=C_i$ for one $j$. Thus, $C_i\subseteq D_j$ up to a null-set. By interchanging the roles of $(C_n)$ and $(D_n)$ the result follows.
\end{proof}

\section{Nodal domains}\label{section:nodal_domains}
In this section we will prove the main theorem, Theorem~\ref{NodalDom}. 
 
Let $Q$ be a positivity preserving form with compact resolvent and let $f$ be a representative of an eigenvector with eigenvalue $\lambda$ for the remainder of this section. 
Define \[F^+:=\{x\in X\colon f(x)> 0\}\] and \[F^-:=\{x\in X\colon f(x)< 0\}.\] 

Our idea of nodal domains is the following: We restrict the form $Q$ to $F^\pm$ and apply Proposition~\ref{Decomp} to get a decomposition of $F^\pm$ in
connected components. These will then be called nodal domains. In order to do so, we first need to define a restriction of $Q$ to $F^\pm$.
Specifically we have two types of restrictions in mind which we want to treat simultaneously. In some very vague sense these restrictions carry parallels to putting Neumann- and Dirichlet boundary conditions on  $F^\pm$.
\begin{itemize}
\item[(a)] The first restriction $Q_{F^\pm}$ is taking all functions in $D(Q)$ that vanish almost everywhere outside of $F^\pm$.
\item[(b)] Secondly, if $Q$ is a regular Dirichlet form and $F^\pm$ is open (see Section~\ref{section:regular}) a different restriction is more natural, namely 
one takes the closure of $D(Q)\cap C_c(F^\pm)$ as domain for the restricted form.
\end{itemize}  
These two approaches are both captured by the following considerations.

Let $\mathcal{A}:=(A_n)_{n\in\IN}$ be an increasing sequence of measurable subsets 
of $X$.  Such a sequence will be called \emph{nest}.  The  \emph{restriction $Q_{F^\pm}^\mathcal{A}$ of $Q$ to $F^\pm$ with respect to $\mathcal{A}$}  is defined on $$D(Q_{F^\pm}^{\mathcal{A}})=\overline{\bigcup_{n=1}^\infty D(Q_{F^\pm\cap A_n})},$$ where the closure is taken in $(D(Q),\|\cdot\|_Q)$ where $ \|\cdot\|_{Q}^{2}=Q(\cdot)+\|\cdot\|^{2}_{L^{2}(X,m)}$.  For  details on restrictions of forms with respect to nests see Appendix~\ref{section:parts_of_forms}. 

For different representatives of $f$ one may get different sets $F^\pm$. 
However, these sets differ only by sets of measure zero. Therefore, the $L^2$-spaces on these sets are equal. Moreover, the forms $Q_{F^\pm\cap A_n}$ are equal, as well, as can be seen by the definition of $D(Q_{F^\pm\cap A_n})$. Thus, by definition,  we infer that the form $Q_{F^\pm}^{\mathcal{A}}$ does not depend on the choice of the representative of $f$.  Therefore, $Q_{F^\pm}^{\mathcal{A}}$ can be seen as a restriction with respect to a nest to the equivalence class of sets that are equal to $F^\pm$ up to a set of measure zero.
 
%\begin{lemma}\label{lemma:generalized_restr_reolvent}
%Let $\mathcal{A}:=(A_n)_{n\in\IN}$ be an increasing sequence of subsets of $X$.
%Let $P_{\mathcal{A}\pm}$ be the orthogonal projection from $D(Q)$ to $D(Q_{F^\pm}^{\mathcal{A}})$. Then, the resolvent $G_\alpha^{\mathcal{A}\pm},$ $\alpha>0$, 
%is compact and the equality \[G_1^{\mathcal{A}\pm} =P_{\mathcal{A}\pm} G_1 \] holds  on $\ltwo$. 
%Moreover, for every $u\in D(Q)$, $u\geq 0$, one has $$P_{\mathcal{A}\pm} u\leq u.$$
%\end{lemma}
%\begin{proof}
%The equality $G_1^{\mathcal{A}\pm} =P_{\mathcal{A}\pm} G_1$ can be seen by the same arguments as used in the proof of Lemma~\ref{ProjectRes}.
%
%The resolvent is compact as the product of a bounded and a compact operator. 
%
%For the inequality let $u\geq0$ be an arbitrary element of $D(Q)$. We show $u\wedge P_{\mathcal{A}\pm} u\in D(Q_{F^\pm}^{\mathcal{A}})$,
%then the inequality follows with the same arguments as in the proof of Lemma~\ref{ProjectIneq}. We only discuss $Q_{F^+}^{\mathcal{A}}$. 
%
%Let $(v_n)$ be a sequence in $\bigcup_{n=1}^\infty D(Q_{A_n^+})$ such that $\|v_n-P_{\mathcal{A}+} u\|_{Q_{F^+}^{\mathcal{A}}}\to 0,$ $n\to\infty$. 
%Then, we get $v_k\wedge u\in \bigcup_{n=1}^\infty D(Q_{A_n^+})$ for every $k\in\IN$, as $u\geq 0$ and every $v_k$ vanishes outside of some set $A_m^\pm$.  
%Moreover, we have $$v_n\wedge u\to P_{\mathcal{A}+} u\wedge u\quad\text{ in }L^2(X,m)$$ and 
%$$\sup_{n\in\IN}\|u\wedge v_n\|_Q\leq\sup_{n\in\IN}(2\|u\|_Q+2\|v_n\|_Q)<\infty.$$ Using Banach-Saks the result follows.
%\end{proof}

 Our next goal is to show that under the assumption $f^\pm\in D(Q_{F^\pm}^{\mathcal{A}})$ we can apply Proposition \ref{Decomp} to the forms $Q_{F^+}^{\mathcal{A}}$ and $Q_{F^-}^{\mathcal{A}}$,
 where for a function $g\in L^2(X,m)$ we define $g^+=g\vee 0$ and $g^-=(-g)\vee 0$. 
\begin{lemma}
 Let $Q$ be a positivity preserving form with compact resolvent and let $f$ be an eigenvector with eigenvalue $\lambda$.
 Let $\mathcal{A}$ be a nest such that $f^\pm\in D(Q_{F^\pm}^{\mathcal{A}})$.
 Then, the following inequality holds
\[G_1^{\mathcal{A},F^\pm} f^\pm\geq \frac{1}{1+\lambda}f^\pm.\]
\end{lemma}
\begin{proof}
We only show the inequality for $f^+$, the other one follows analogously.

By Lemma~\ref{lemma:closed:eigenvalue_resolvent} we have $G_1 f=\frac{1}{\lambda+1}f$. Moreover, using the definition
 $G_1^{\mathcal{A},F^+}f:=G_1^{\mathcal{A},F^+}(1_{F^+}f)$ we infer $G_1^{\mathcal{A},F^+} f=G_1^{\mathcal{A},F^+} f^+$. 
 Furthermore, the equality  $G_1^{\mathcal{A},F^+} f^+=P_{\mathcal{A},F^+}G_1 f$ holds by Proposition~\ref{ProjectRes}, where $ P_{\mathcal{A},F^+} $ is the orthogonal projection from $ D(Q) $ to $ D(Q^{A}_{F^{+}}) $. Thus, one has 
 \begin{align*}G_1^{\mathcal{A},F^+} f^+&=G_1^{\mathcal{A},F^+} f=P_{\mathcal{A},F^+}G_1 f=\frac{1}{\lambda+1}P_{\mathcal{A},F^+}f\\&=\frac{1}{\lambda+1}P_{\mathcal{A},F^+}f^+-\frac{1}{\lambda+1}P_{\mathcal{A},F^+}f^-.\end{align*}
 Finally, Proposition~\ref{ProjectIneq} yields $P_{\mathcal{A},F^+}f^-\leq f^-= 0$ on $F^+$ and, therefore, 
 \[G_1^{\mathcal{A},F^+} f^+\geq \frac{1}{\lambda+1}P_{\mathcal{A},F^+}f^+=\frac{1}{\lambda+1}f^+,\] since $f^+\in D(Q_{F^+}^{\mathcal{A}})$ by assumption.
 This concludes the proof. 
\end{proof}
Thus, we get the following corollary.
\begin{corollary}
 Let $Q$ be a positivity preserving form with compact resolvent and let $f$ be a representative of an eigenvector with eigenvalue $\lambda$.  Let $\mathcal{A}$ be a nest such that $f^\pm\in D(Q_{F^\pm}^{\mathcal{A}})$. If  $m(F^+)>0$ (respectively $m(F^-)>0$), then, $F^+$ (respectively $ F^{-} $) can be uniquely decomposed (up to sets of measure zero) into finitely many $Q_{F^+}^{\mathcal{A}}$-connected components (respectively $Q_{F^.}^{\mathcal{A}}$-connected components).
\end{corollary}   
% Let $Q$ be a positivity preserving form with compact resolvent and let $f$ be a representative of an eigenvector with eigenvalue $\lambda$.
As discussed above, with a slight abuse of notation we will not distinguish between $Q_{F^\pm}^{\mathcal{A}}$-connected components and their equivalence classes with respect to the equivalence relation 
$A\sim B:\aeq m(A\Delta B)=0$ on measurable subsets of $X$. In this sense, $F^\pm$ can be uniquely decomposed into finitely many $Q_{F^\pm}^{\mathcal{A}}$-connected components.

Now suppose we have two different representatives of an eigenvector at hand. Then, as discussed earlier, both yield the same form $Q_{F^\pm}^{\mathcal{A}}$. Thus, both yield the same $Q_{F^\pm}^{\mathcal{A}}$-connected components. Hence, the following notion of nodal domains is well-defined. 

\begin{definition}
Let $Q$ be a positivity preserving form with compact resolvent and let  $f$ be an eigenvector with eigenvalue $\lambda$. 
Let $\mathcal{A}$ be a nest such that $f^\pm\in D(Q_{F^\pm}^{\mathcal{A}})$.
Then, we call the $Q_{F^\pm}^{\mathcal{A}}$-connected components \emph{$(Q,\mathcal{A})$-nodal domains}. Moreover,  
we call every $Q_{F^+}^{\mathcal{A}}$-connected component a \emph{positive} $(Q,\mathcal{A})$-nodal domain and every $Q_{F^-}^{\mathcal{A}}$-connected component a \emph{negative} $(Q,\mathcal{A})$-nodal domain.
\end{definition} 
Recall, that a connected component is assumed to have positive measure, so if either of the sets $F^+$ or $F^-$ is of measure zero, there are only negative or only positive $(Q,\mathcal{A})$-nodal domains.
\begin{lemma}\label{lemma:nodal_dom_Q_positive}
Let $Q$ be a positivity preserving form with compact resolvent.  Let $\lambda$ be an eigenvalue of $Q$ and $f$ be an eigenvector for $\lambda$. 
Let $\mathcal{A}$ be a nest such that $f^\pm\in D(Q_{F^\pm}^{\mathcal{A}})$.
Let $C_1,\ldots,C_l$ be the $(Q,\mathcal{A})$-nodal domains of $ f $.
 Then, $f1_{C_1},\ldots,f_{C_{l}}\in D(Q)$ and \[f=\sum_{i=1}^l f1_{C_i}\] holds. Furthermore, one has \[Q(f1_{C_i},f1_{C_j})\geq 0,\qquad i,j=1,\ldots,l.\] 
\end{lemma}
\begin{proof} At first we show $f1_{C_i}\in D(Q)$. 
Without loss of generality let $C_i$ be a positive $(Q,\mathcal{A})$-nodal domain.
 Then, $f1_{C_i}=f^+1_{C_i}$ is an element of $D(Q_{F^+}^{\mathcal{A}})$  by $Q_{F^+}^{\mathcal{A}}$-invariance of $C_i$ and since $f^+\in D(Q_{F^+}^{\mathcal{A}})$. 
 Hence, $f1_{C_i}$ is an element of $D(Q)$. 
  
  Note, that $f=\sum_i f1_{C_i}$ holds as $F^+$ and $F^-$ can be decomposed (up to a set of measure zero) into the $(Q,\mathcal{A})$-nodal domains and $f$ vanishes outside of 
  $F^+\cup F^-$ $m$-almost everywhere.
  
Finally we show $Q(f1_{C_i},f1_{C_j})\geq 0$ for all $i,j$. This is clear for $ i=j $. Moreover, if $C_i$ and $C_j$ have the same sign and $i\not=j$, 
  then by invariance one has $Q(f1_{C_i},f1_{C_j})= 0$. Otherwise, without loss of generality let $C_j$ be a negative
  and $C_i$ be a positive $(Q,\mathcal{A})$-nodal domain, i.e., $f=f^+$ on $C_i$ and $f=-f^-$ on $C_j$. Then, one has, since $C_i$ and $C_j$ are disjoint and 
  $Q$ is positivity preserving,
  \begin{align*}
%  Q(f1_{C_i})-2Q(f1_{C_i},f1_{C_j})+Q(f1_{C_j})    =&
    Q(f1_{C_i}-f1_{C_j})
    =&Q(f^+1_{C_i}+f^-1_{C_j})\\
	 =&Q(|f1_{C_i}+f1_{C_j}|)\\  
    \leq&Q(f1_{C_i}+f1_{C_j})
 %   =&Q(f1_{C_i})+2Q(f1_{C_i},f1_{C_j})+Q(f1_{C_j}).
\end{align*}  
Now, $ Q(f1_{C_i}\pm f1_{C_j})=Q(f1_{C_i})\pm2Q(f1_{C_i},f1_{C_j})+Q(f1_{C_j}) $ yields the statement.
  \end{proof}
  
Next we give a technical lemma, which is essential for the proof of Theorem~\ref{NodalDom} below. It is a direct generalization of \cite[Lemma 1]{Davies}.
\begin{lemma}\label{SumLemma}
Let $Q$ be a positivity preserving form with compact resolvent.  Let $\lambda$ be an eigenvalue of $Q$ and $f$ be an eigenvector for $\lambda$. 
Let $\mathcal{A}:=(A_n)_{n\in\IN}$ be a nest such that $f^\pm\in D(Q_{F^\pm}^{\mathcal{A}})$.
Let $C_1,\ldots,C_l$ be the $(Q,\mathcal{A})$-nodal domains of $f$. Let $c_1,\ldots,c_k\in\IR$ be arbitrary and set $$v=\sum_{k=1}^l c_k f1_{C_k}.$$
Let $\mu>0$ be arbitrary. Then, the following equality holds
\begin{align*}
Q(v)-\mu\|v\|_\ltwo^2=&\sum_{i=1}^l c_i^2 (Q(f1_{C_i},f)-\mu\| f1_{C_i}\|_\ltwo^2)\\&-\frac 12\sum_{i,j=1}^l(c_i-c_j)^2Q(f1_{C_i},f1_{C_j}) .
\end{align*}
\end{lemma}
\begin{proof}
One has by disjointness (up to sets of measure zero) of $C_1,\ldots,C_l$\[-\mu\|v\|_\ltwo^2=-\mu\sum_{i=1}^l c_i^2\| f1_{C_i}\|_\ltwo^2=-\mu\sum_{i=1}^l \| c_i f1_{C_i}\|_\ltwo^2.\] Furthermore, the equality 
\begin{align*}
\frac{1}{2}\sum_{i,j=1}^l(c_i-c_j)^2Q(f1_{C_i},f1_{C_j})
&=\sum_{i,j=1}^l c_i^2Q(f1_{C_i},f1_{C_j})-\hspace{-.2cm}\sum_{i,j=1}^l c_i c_jQ(f1_{C_i},f1_{C_j})\\
&=\sum_{i=1}^l c_i^2Q( f1_{C_i},f)-Q(v)
\end{align*} 
holds. This concludes the proof.
\end{proof}      
Next we give the main theorem of this paper. The proof follows the ideas of the classical nodal domain theorem for graphs, c.f. \cite{Davies}.
\begin{theorem}[Courant nodal domain theorem]\label{NodalDom}
Let $Q$ be a positivity preserving form with compact resolvent. Let $\lambda_1,\lambda_2,\ldots$ be the eigenvalues of $Q$ in ascending order counted with multiplicity. Let $f$ be an eigenfunction for the eigenvalue $\lambda_n$ with   multiplicity  $ k $ such that $f^\pm\in D(Q_{F^\pm}^{\mathcal{A}})$ for a nest $\mathcal{A}$.
Then, $f$ has at most $(n+k-1)$ $(Q,\mathcal{A})$-nodal domains. 
\end{theorem}
\begin{proof} Let $C_1,\ldots,C_l$ be the $(Q,\mathcal{A})$-nodal domains of $f$. Then, by definition of $(Q,\mathcal{A})$-nodal domains, $f1_{C_k}\not= 0$ and $\langle f1_{C_i},f1_{C_j}\rangle_\ltwo=0$ for all $k,i,j$, $i\not=j$.
 Therefore, the functions $f1_{C_1},\ldots,f1_{C_l}$ are orthogonal elements of $\ltwo$ and their span has dimension $l$. Now, choose $c_1,\ldots,c_l$ such that $v=\sum_{i=1}^l c_if1_{C_i}$
 is normalized and orthogonal to the $(l-1)$-dimensional span of eigenfunctions for $\lambda_1,\ldots,\lambda_{l-1}$. The  variational principle (Lemma~\ref{lemma:closed:varia_eigen}) yields \[Q(v)\geq \lambda_l.\]
  
  Observe, that \[Q(1_{C_i}f,f)-\lambda_n\|f1_{C_i}\|_\ltwo^2=\lambda_n(\langle f1_{C_i},f\rangle_\ltwo-\|f1_{C_i}\|_\ltwo^2)=0\] holds for every $i=1,\ldots l$. By $\|v\|_\ltwo=1$,  Lemma \ref{SumLemma} applied  with $\mu=\lambda_n$ yields \[Q(v)-\lambda_n=-\frac 12\sum_{i,j=1}^l(c_i-c_j)^2Q(f1_{C_i},f1_{C_j}).\]  Since by Lemma~\ref{lemma:nodal_dom_Q_positive} we know $Q(f1_{C_i},f1_{C_j})\geq 0$ for every $i,j$, we infer  
\[Q(v)-\lambda_n\leq 0\] and, in summary, $$\lambda_l\leq Q(v)\leq \lambda_n.$$ This yields \[\lambda_{n+k}>\lambda_n\geq \lambda_l,\] where the first inequality holds as $\lambda_n$ is the $n$-th eigenvalue and 
has multiplicity $k$. Since the eigenvalues are ordered by size, we get $l<n+k$, and, hence, $l\leq n+k-1$.
\end{proof}

\begin{remarks}
	The theorem above includes the setting of Schrödinger forms as it is described in \cite[Section~4]{StVoi96}. We refrain at this point from discussing explicit examples in this generality but we take a closer look at Dirichlet forms in the next section.
\end{remarks}

To end this section, we discuss the special case where the nest $\mathcal{A}$ is trivial in the sense $A_n=X$ for every $n$. Let $f$ be an eigenvector for $\lambda$. Then, we infer 
 $$Q_{F^\pm}^\mathcal{A}= Q_{F^\pm}$$ and, hence, $f^\pm\in D(Q_{F^\pm}^\mathcal{A})$. Therefore, the results above yield the existence of $(Q,\mathcal{A})$-nodal domains which we call \emph{$Q$-nodal domains} in this case. The preceding theorem immediately yields the following corollary.
\begin{corollary}\label{corollary:courant_Q_nodal}
Let $Q$ be a positivity preserving form with compact resolvent. Let $\lambda_1,\lambda_2,\ldots$ be the eigenvalues of $Q$ in ascending order counted with multiplicity. Every eigenfunction for the eigenvalue  $\lambda_n$ with 
 multiplicity $ k $ has at most $(n+k-1)$ $Q$-nodal domains. 
\end{corollary}
In Section~\ref{section:graphs} we will see an example in which the $Q$-nodal domains coincide with the classical ones.
\section{Regular Dirichlet forms}\label{section:regular}
This chapter consists mainly of two classes of examples, namely Dirichlet forms on graphs and local Dirichlet forms. These will show that the notion of $(Q,\mathcal{A})$-nodal domains generalizes the classical ones for elliptical operators on $\IR^n$ and Laplacians on finite graphs.

In this whole section we make the standing assumption that we are given a locally compact, separable metric space $X$ equipped with a Radon measure $m$ on $\mathcal{B}(X)$ of full support. Furthermore, for a form $ Q $ with compact resolvent we denote the eigenvalues counted with multiplicity by $$  \lm_{1},\lm_{2},\ldots  $$ counted with multiplicity. We first introduce regular Dirichlet forms.

\begin{definition}\label{def:dirichlet_form}
 A positivity preserving form $Q$ on $D(Q)\subseteq L^2(X,m)$ is called \emph{Dirichlet form} if for every $u\in D(Q)$ we have $0\vee u\wedge 1\in D(Q)$ and 
 \[Q(0\vee u\wedge 1)\leq Q(u).\] Let $C_c(X)$ be the space of all compactly supported, continuous, real valued functions.
 A Dirichlet form is called \emph{regular} if $C_c(X)\cap D(Q)$ is both $\|\cdot\|_\infty$-dense in $C_c(X)$ and $\|\cdot\|_Q$-dense in $D(Q)$. 
\end{definition}

Note that the restriction $ Q_{A} $ as discussed above is not necessarily a \emph{regular} Dirichlet form. Therefore, we introduce a different type of restriction for which the regularity of the restricted form can be seen.

\begin{definition}
Let $Q$ be a regular Dirichlet form and let $A\subseteq X$ be open. Then, we define $Q^A$ as the restriction of $Q$ to $$D(Q^A):= \overline{C_c(A)\cap D(Q)}^{\|\cdot\|_Q}.$$
We call $Q^A$ the \emph{part of $Q$ on $A$}.
\end{definition}

For this type of restriction of regular Dirichlet forms the following useful lemma holds.
\begin{lemma}[\cite{Fuku}, Theorem 4.4.3]\label{lemma:reg:restr_open_reg}
 Let $Q$ be a regular Dirichlet form and  $\emptyset\not=A\subseteq X$ be open. Then,  $Q^A$ is a regular Dirichlet form on $L^2(A,m)$. 
\end{lemma}
\begin{remarks}
It is possible that the equality $Q_A=Q^A$ holds. Sets with this property are called Kac-regular for $Q$. More informations on this property can be found in \cite{Stroock67, HZ87, BG17, W17}  Specifically, for regular Dirichlet forms on graphs, we will see in the next section that every set is Kac-regular.
\end{remarks}

Suppose we have a continuous eigenfunction $f$ of $Q$ to our disposal. Then, $F^\pm$ are open sets and, hence, $Q^{F^+}$ and $Q^{F^-}$ are regular Dirichlet forms. 
Since regular Dirichlet forms are quasi-regular, c.f. \cite[Section 1.3.]{ChenFuku}, there are nests $(A_n^\pm)$ of subsets of $F^\pm$ such that 
$D(Q^{F^\pm})=\overline{\bigcup_{n=1}^\infty D(Q_{A_n^\pm})}^{\|\cdot\|_Q}$. We define $\mathcal{A}$ via $A_n=A_n^+\cup A_n^-$.
Then, we infer $Q^{F^\pm}=Q_{F^\pm}^{\mathcal{A}}$.

Moreover, we deduce $f^\pm\in D(Q^{F^\pm})$ by continuity of $f^\pm$ and \cite[Theorem 4.4.3]{Fuku}. 
Therefore, this type of restriction is a special case of the one in Section~\ref{section:nodal_domains}. 
In particular, the $Q^{F^\pm}$-connected components are equal to the $(Q,\mathcal{A})$-nodal domains. 
 We call these representatives  
\emph{Dirichlet nodal domains.}
 
 We get the following corollary. It is a direct consequence of Theorem~\ref{NodalDom}. 
\begin{corollary}\label{corollary:reg:NodalDom}
Let $Q$ be a regular Dirichlet form with compact resolvent. 
 Then every  continuous eigenfunction   for the eigenvalue $\lambda_n$ with multiplicity $ k $  has at most $(n+k-1)$ Dirichlet nodal domains. 
\end{corollary}

\subsection{Dirichlet forms on graphs}\label{section:graphs}
A Radon measure $ m $ of full support on a finite or countably infinite non-empty set $ X$ equipped with the discrete topology is uniquely determined by a function $m:X\to (0,\infty)$ and we denote the corresponding space of square-summable functions by $\ell^2(X,m)$.
A \emph{graph}  $(b,c)$ over $ X $ is given by of functions $b:X\times X\to[0,\infty),c:X\to[0,\infty)$ such that $b$ 
is symmetric, vanishes on the diagonal and satisfies
 \[\sum_{y\in X} b(x,y)<\infty,\qquad x\in X.\]
We call $X$ the \emph{vertex set}, $b$ the \emph{edge weight} and $c$ the \emph{killing term}. The graph  is called \emph{connected} if for all $ x,y\in X $, $ x\neq y $, there is a sequence of vertices $ x=x_{0},\ldots,x_{n+1}=y $ such that
$b(x_i,x_{i+1})>0$  for all $ i=0,\ldots,n $. 

In \cite{StochCompl} it is shown that for every  a regular Dirichlet form  $Q$ on $\ell^2(X,m)$ there is a graph $ (b,c) $ over $ X $ such that for all  $ f\in D(Q) $
 $$Q(f)=\frac12\sum_{x,y\in X} b(x,y)(f(x)-f(y))^2+\sum_{x\in X} c(x)f(x)^2$$
 and vice versa every graph yields a regular Dirichlet form as above. 
 We say $ (b,c) $ and $ Q $ are \emph{associated}. 
 Furthermore, the functions of finite support $C_c(X)$ are included in $D(Q)$.

Given a subset $A\subseteq X$ we can define a graph $(b_A,c_A)$ on $A$ via \[b_A(x,y)=b(x,y), \quad c_A(x)=c(x)+\sum_{z\in X\setminus A} b(x,z)\] for $x,y\in A$.
We say that $A$ is \emph{connected} if the graph $(b_A,c_{A})$ is connected. We call $A$ a \emph{graph connected component} of $X$ if $A$ is a maximal (with respect to inclusion) connected subset of $X$.  
Let $C_c(A)$ be subspace  of all  functions in $ C_c(X) $ supported on $A$. 
The part  $Q^A$ of $Q$ on $A$ with domain
 $$ D(Q^A)=\overline{C_c(A)}^{\|\cdot\|_Q} $$ is a regular Dirichlet form and it is associated to the graph $ (b_{A},c_{A}) $.
Using \cite[Theorem 4.4.3]{Fuku} and the fact the every function on a discrete set is continuous, we infer $D(Q_A)=D(Q^A)$, and, hence, $$ Q_A=Q^A. $$

Our goal is to show that the Dirichlet nodal domains of an eigenfunction are exactly the graph nodal domains, i.e., the graph connected components of $F^+$ and $F^-$, respectively. For this, we first 
characterize $Q$-invariance.
\begin{lemma}[$ Q $-invariance]\label{lemma:graph:char_invariance} 
 Let $A\subseteq X$. Then, $A$ is $Q$-invariant if and only if $A$ is the union of graph connected components of $X$.
\end{lemma}
\begin{proof}
 Assume that $A$ is not the union of graph connected components of $X$. Then, there are $x\in A$, $y\in X\setminus A$ such that $b(x,y)>0$. Then, the function 
 $1_x+1_y\in D(Q)$ satisfies $(1_x+1_y)1_A=1_x\in D(Q)$  and  $(1_x+1_y)1_{X\setminus A}=1_y\in D(Q)$. We calculate \[Q((1_x+1_y)1_A,(1_x+1_y)1_{X\setminus A})=Q (1_x,1_y)=b(x,y)\not= 0\] and, hence, $A$ is not invariant.
 
 Now suppose $A$ is a union of graph connected components of $X$. We infer $b(x,y)=0$ for all $x\in A,y\in X\setminus A$. 
 Therefore, we get the equality $D(Q)=D(Q_A)\oplus D(Q_{X\setminus A})$, where the sum is direct with respect to $Q$. %This concludes the proof.
\end{proof}
Next we will characterize irreducibility of $Q_A$.
\begin{lemma}[Irreducibility]\label{lemma:graph:char_irreducible} 
 Let $A\subseteq X$. Then, the form $Q_A$ is irreducible if and only if $A$ is connected. 
\end{lemma}
\begin{proof}
 Let $Q_A=Q^A$ be irreducible. Thus, $A$ is the only non-empty $Q_A$-invariant subset of $A$. 
 Since, by  Lemma \ref{lemma:graph:char_invariance} applied to $Q^A$, every graph connected component of $A$ is $Q^A$-invariant, we infer that $A$ is a graph 
 connected component of $A$ and, thus, $A$ is connected.
 
 On the other hand, let $A$ be connected. Then, $A$ is the only graph connected component of $A$. But, since by Lemma \ref{lemma:graph:char_invariance} applied to $Q^A$
 the $Q^A$-invariant sets are the unions of graph connected components of $A$, the only non-empty $Q^A$-invariant set is $A$. 
 Therefore, $Q^A$ is irreducible.
\end{proof}
Now we are able to characterize Dirichlet nodal domains.
\begin{theorem}[Dirichlet nodal domains vs. graph nodal domains]\label{theorem:graph:Dirichlet_vs_graph}
 Let $(b,c)$ be a graph over   $(X,m)$ and assume that $Q$ has compact resolvent.
 Let $f$ be an eigenfunction of $Q$. Then, the Dirichlet nodal domains of $f$ are exactly the graph connected components of $F^+=\{f>0\}$ and $F^-=\{f<0\}$.
\end{theorem}
\begin{proof}
 We only discuss $F^+$. By definition, the positive Dirichlet nodal domains of $f$ are the $Q_{F^+}$-invariant subsets $A$ such that the restricted form 
 $(Q_{F^+})_A=Q_A=Q^A$ is irreducible.
 By Lemma~\ref{lemma:graph:char_irreducible} the irreducibility of $Q^A$ is equivalent to the connectedness of $A$. 
 Furthermore, by Lemma~\ref{lemma:graph:char_invariance} the 
 $Q_{F^+}$-invariance is equivalent to $A$ being the union of $F^+$-connected components. Hence, $A$ is a positive Dirichlet nodal domain if and only 
 if it is connected and the union 
 of graph connected components of $F^+$.  This is equivalent to $A$ being a graph connected component of $F^+$. 
\end{proof}
The previous theorem can be used to give the Courant bound for  \emph{graph nodal domains}, i.e. graph connected components of $F^+$ or $F^-$.
\begin{corollary}[Courant bound for graphs]
 Let $(b,c)$ be a graph over $(X,m)$ and assume that $Q$ has compact resolvent.
 Then, every eigenfunction  for the eigenvalue $\lambda_n$ with multiplicity $ k $   has at most $(n+k-1)$ graph nodal domains. 
\end{corollary}
\begin{proof}
 This follows easily using the previous theorem and Corollary~\ref{corollary:reg:NodalDom}.
\end{proof}

Finally, we  give two examples for which the corollary above can be applied.
\begin{example}  For finite sets $ X $ this result was proven in \cite{Davies}. 
\end{example}
\begin{example}
 Let $X$ be an infinite set. We call a graph $(b,c)$ over $X$ uniformly transient if there is a $C>0$ such that \[\|f\|_\infty^2\leq CQ(f), \qquad f\in C_{c}(X).\]
%This class of graphs was introduced and studied in \cite{UnifTrans}. %An example of such a graph is the infinite binary tree.
Equipping such graph with a finite measure, the form $Q$ has compact resolvent by  \cite[ Theorem~7.2]{UnifTrans}.
\end{example}
To end this subsection about graphs, we want to discuss an application of Corollary~\ref{corollary:courant_Q_nodal} to a non-regular form on graphs. Let $(b,c)$ be a graph over $(X,m)$. We define a quadratic form $\widetilde{Q}:C(X)\to[0,\infty]$, $$\widetilde{Q}(f)=\frac12\sum_{x,y\in X} b(x,y)(f(x)-f(y))^2+\sum_{x\in X} c(x)f(x)^2$$ and $$\widetilde{D}=\{f\in C(X)\colon \widetilde{Q}(f)<\infty\}.$$
Then, the restriction $Q^{(N)}$ of $\widetilde{Q}$ to $D(Q^{(N)})=\widetilde{D}\cap \ell^2(X,m)$ is a Dirichlet form. However, in general $Q^{(N)}$ is not regular, \cite{StochCompl}.
Let $A$ be a subset of $X$. Then, there is a quadratic form $\widetilde{Q}_A$ and  $\widetilde{D}_A$ associated to the graph $(b_A,c_A)$. Denote by $Q_{A}^{(N)}$ the restriction of  $Q^{(N)}$ to 
$\widetilde{D}_A\cap\ell^2(A,m)$. Then, a direct calculation shows the equality $$Q_A^{(N)}=(Q^{(N)})_A.$$ %We get the following lemma. 
\begin{lemma}\label{lemma:graph:char_Neumann_invariance} 
 Let $A\subseteq X$. Then, $A$ is $Q^{(N)}$-invariant if and only if $A$ is the union of graph connected components of $X$.
Moreover, the form $(Q^{(N)})_A$ is irreducible if and only if $A$ is connected. 
\end{lemma}
\begin{proof}
 This follows analogously to Lemma~\ref{lemma:graph:char_invariance} and Lemma~\ref{lemma:graph:char_irreducible}.
\end{proof}
A graph is called \emph{canonically compactifiable}, if the inclusion $\widetilde{D}\subseteq \ell^\infty(X)$ holds. It is known, \cite[Corollary 5.2]{GHKLW15}, that the form $Q^{(N)}$ on a canonically compactifiable graph with finite measure has compact resolvent. 
\begin{proposition}
Let $(b,c)$ be a canonically compactifiable graph over $ (X,m) $ and let $f$ be an eigenfunction of $Q^{(N)}$. Then, the $Q^{(N)}$-nodal domains of $f$ are exactly the graph connected components of $F^+$ and $F^-$.
\end{proposition}
\begin{proof}
This follows analogously to Theorem~\ref{theorem:graph:Dirichlet_vs_graph}.
\end{proof}

\begin{corollary}[Courant bound for graphs]
 Let $(b,c)$ be a canonically compactifiable graph  over $ (X,m) $.
 Then every eigenfunction  for the eigenvalue $\lambda_n$ with multiplicity $ k $  has at most $(n+k-1)$ graph nodal domains. 
\end{corollary}
\begin{proof}
 This follows  from the previous proposition and Corollary~\ref{corollary:courant_Q_nodal}.
\end{proof}

\subsection{Local Dirichlet forms}\label{section: local}
Here we apply our results to so-called local Dirichlet forms. The classical examples are the Dirichlet forms $Q(u)=\int_\Omega |a\nabla u|^2 dx$ with measurable coefficients $ a $ on an open subset 
$\Omega\subseteq \IR^n$.

In the first subsection we will show that in a rather general setting a topological nodal domain can be decomposed into Dirichlet nodal domains, 
where for a continuous eigenfunction $f$ a \emph{topological nodal domain} is a connected component (in the topological sense) of $\{f>0\}$ or $\{f<0\}$.  
In the second and third subsection we will look at two classes of examples, for which we show  that the classical notion of nodal domains coincides with the notion of 
Dirichlet nodal domains.
\begin{definition}
A Dirichlet form $Q$ is called \emph{local} if for every $u,v\in D(Q)$ with $u\cdot v=0$ we have $ Q(u,v)=0 $.  
\end{definition}
\begin{remarks}\label{remark:reg:restr_local_is_local}
(a) Obviously a set $A$ is invariant for a local, regular Dirichlet form if and only if we have $1_Af\in D(Q)$ for every $f\in D(Q)$. \\
(b) Using Lemma~\ref{lemma:reg:restr_open_reg} one deduces easily that the part of a local, regular Dirichlet form on an open set is again a local, 
  regular Dirichlet form. 
\end{remarks}

\subsubsection{The general case}\label{section:general_local}
We show next that under certain conditions a topological nodal domain can be decomposed into Dirichlet nodal domains. 
We start with the following 
lemmas. 

\begin{lemma}\label{lemma:local:connected_comp_inv}
 Let $Q$ be a local, regular Dirichlet form on a locally connected space $ X $. Then every topological connected component of $X$ is $Q$-invariant. 
\end{lemma}
\begin{proof}
 Since the space is locally connected, we infer that topological connected component is simultaneously open and closed. Hence, we can apply \cite[Corollary 4.6.3]{Fuku} and the result follows.
\end{proof}
\begin{lemma}\label{lemma:local:equality_of_restrictions}
Let $Q$ be a local, regular Dirichlet form on a locally connected topological space $X$. Then, for every topological connected component $ B $ of an open set $A\subseteq X$, we have $$(Q^A)_B=Q^B.$$ 
\end{lemma}
\begin{proof}
By local connectedness,  $B$ is an open subset of $A$. By Lemma~\ref{lemma:local:connected_comp_inv} the set $B$ is $Q^A$-invariant.
We have to show $D((Q^A)_B)=D(Q^B)$.  

"$ \supseteq $": Using the openness of $B$, we get $C_c(B)\subseteq C_c(A)$ and, hence, $D(Q^B)\subseteq D(Q^A)$. Let $u\in D(Q^B)$. We show $u1_{A\setminus B}=0$ to infer $u\in D((Q^A)_B)$.
Let $(u_n)$ be a sequence in $C_c(B)$ such that $\|u_n-u\|_Q\to 0$. In particular, $ \|u_n-u\|_{L^2(X,m)}\to 0$ and, since every $u_n$ vanishes outside of $B$, we obtain
$u1_{A\setminus B}=0$.

"$ \subseteq $":  Let $u\in D((Q^A)_B)$ be arbitrary and $(u_n)$ be a sequence in $C_c(A)$ such that $\|u_n-u\|_Q\to 0$. Hence,  $u_n1_B\in C_c(B)$ and by $Q^A$-invariance of $B$ we infer $u_n1_B\in D(Q_A)$ for every 
$n$. Using locality of $Q$, we  estimate 
\begin{align*}
\|u-u_n1_B\|_Q^2\leq \|u-u_n1_B\|_Q^2+\|u_n1_{A\setminus B}\|_Q^2=\|u-u_n\|_Q^2\to 0,
\end{align*}
as $ n\to\infty. $
This yields $u\in D(Q^B)$.
\end{proof}
Using Lemma~\ref{lemma:local:connected_comp_inv}, we get the desired result.

\begin{theorem}[Dirichlet nodal domains vs topological nodal domains]\label{proposition:local:connected_comp_decomp_in_nodal_domains}
 Let $Q$ be a local, regular Dirichlet form on a locally connected space $ X $ with compact resolvent. Let $f$ be a continuous eigenfunction of $Q$ and let $A$ be 
 a topological connected component of $F^+$ or $F^-$ (i.e., a classical nodal domain). Then, there are Dirichlet nodal domains $C_1,\ldots,C_l$ 
 such that $A=\bigcup_{i=1}^l C_i$, where the equality holds up to a set of measure zero.
\end{theorem}
\begin{proof}
Let $A$ be a topological connected component of $F^+$. Since $f$ is continuous, the set $F^+$ is open. 
 Hence, $Q^{F^+}$ is again a local, regular Dirichlet form by Remark~\ref{remark:reg:restr_local_is_local}.
 Since open subsets of locally connected spaces are again locally connected, we can apply Lemma \ref{lemma:local:connected_comp_inv} and infer that $A$ is  
 $Q^{F^+}$-invariant.
 Now let $C_1,\ldots,C_l$ be the Dirichlet nodal domains that satisfy $m(C_i\cap A)>0$ all $ i $. Such Dirichlet nodal domains exist, since $F^+$ can be decomposed 
 into the Dirichlet nodal domains up to an
 $m$-null set and $A$ is an open set in $X$ (since $A$ is open in $F^+$ and $F^+$ is open in $X$) and, therefore, $ A $ has positive measure (since $m$ has full support).  
 It is left to show $C_i\subseteq A$ (up to an $m$-null set). Then, the result follows since $F^+$ can be decomposed into the Dirichlet nodal domains.
 Suppose the contrary, i.e., for an $i\in\{1,\ldots,n\}$ we have $m(C_i\cap(F^+\setminus A))>0$. Then, since the invariant sets form a $\sigma$-algebra, Lemma~\ref{lemma:Inv:Inv_is_sigma_field}, and since $A$ is $Q^{F^+}$-invariant by Lemma~\ref{lemma:local:connected_comp_inv} we infer that both $C_i\cap A$ and 
 $C_i\cap (F^+\setminus A)$ are $Q^{F^+}$-invariant.  But by Lemma \ref{InvTrans} this implies that $ C_i\cap A$ and 
 $C_i\cap (F^+\setminus A)$ are $(Q^{F^+})_{C_i}$-invariant. This is a contradiction to the irreducibility of $(Q^{F_+})_{C_i}$. Hence, the proposition is proven.
\end{proof}
\begin{corollary}[Courant bound for local, regular Dirichlet forms]
 Let $Q$ be a local, regular Dirichlet form on a locally connected space $ X $ with compact resolvent.   
 Then every a continuous eigenfunction for the eigenvalue $\lambda_n$ with multiplicity $ k $ has at most $(n+k-1)$ (topological) nodal domains.
\end{corollary}
\begin{proof}
 By the proposition above every topological nodal domain contains at least one Dirichlet nodal domain. Hence, the number of topological nodal domains is smaller than the one of 
 Dirichlet nodal domains, which is bounded from above by $(n+k-1)$ by Corollary~\ref{corollary:reg:NodalDom}.  
\end{proof}
For general Dirichlet forms, the upper bound of $(n+k-1)$ Dirichlet nodal domains can not be improved. An example (on finite graphs) for this can be found in \cite{Davies}. 
However, the classical nodal domain theorem for the Laplacian gives 
an upper bound of $n$ instead of  $(n+k-1)$. This is because of an unique continuation principle, which holds for the Laplacian in $\IR^n$ but not on graphs. 
We show this sharper 
upper bound under the assumption of a unique continuation principle. We say the form $ Q $ satisfies a\emph{ unique continuation principle} if one of the following two assumptions is satisfied:
\begin{itemize}
\item[(UC1)] Eigenfunctions of $Q$ do not vanish on a set of positive measure. 
\item[(UC2)]  Eigenfunctions of $Q$ do not vanish on non-empty open sets and every Dirichlet nodal domain includes a non-empty open set.
\end{itemize}

 %\begin{definition}
 %Let $Q$ be an irreducible local, regular Dirichlet  form with compact resolvent.
 %\begin{itemize}
 %\item[(UC1)] We say that $Q$ satisfies  \emph{unique continuation principle for sets of positive measure} (UC1) if every eigenfunction
 %of $Q$ that vanishes on a set of positive measure already vanishes almost everywhere on $X$. 
 %\item[(UC2)]  We say that $Q$ satisfies a \emph{unique continuation principle for open sets }(UC2) if every eigenfunction
 %of $Q$ that vanishes on an open set of positive measure already vanishes almost everywhere on $X$. 
%\end{itemize}
%\end{definition}
%\begin{theorem}
 %Let $Q$ be an irreducible local, regular Dirichlet form with compact resolvent. Suppose that $Q$ satisfies either (UC1), or (UC2) and every Dirichlet nodal domain   contains an open subset of positive measure of $X$.  Let $\lambda_1,\lambda_2,\ldots$ be the eigenvalues of $Q$ in ascending order counted with multiplicity. Let $f$ be an eigenfunction for $\lambda_n$. Assume that $f$ is continuous.  Then, $f$ has at most $n$ Dirichlet nodal domains. \end{theorem}
\begin{theorem}
	Let $Q$ be an irreducible local, regular Dirichlet form with compact resolvent that satisfies $ \mathrm{(UC1)} $ or $ \mathrm{(UC1)} $. Then every continuous eigenfunction for $\lambda_n$  has at most $n$ Dirichlet nodal domains. 
\end{theorem}

\begin{proof}
 Assume that $f$ has $l>n$ Dirichlet nodal domains $C_1\ldots,C_l$. Then, $f1_{C_k}\not\equiv 0$ and $\langle f1_{C_i},f1_{C_j}\rangle_\ltwo=0$ for all $i,j$, $i\not=j$.
 Therefore, they are orthogonal elements of $\ltwo$. In particular, the span of $(f1_{C_k})_{k=1,\ldots,n}$ has dimension $n$. Now choose $c_1,\ldots,c_n$ 
 such that $v=\sum_{i=1}^n c_if1_{C_i}$
 is normalised and orthogonal to the $n-1$ eigenfunctions for $\lambda_1,\ldots,\lambda_{n-1}$. The variational characterization, Lemma~\ref{lemma:closed:varia_eigen} yields \[Q(v)\geq \lambda_n.\]
 On the other hand, since the Dirichlet form is local, we infer 
 $$Q(f1_{C_i})=Q(f,f1_{C_i})=\lambda_n\langle f,f1_{C_i}\rangle_{L^2(X,m)}=\lambda_n\|f1_{C_i}\|_{L^2(X,m)}^2.$$ 
 %that for every $u\in D(Q_{C_i})$ the equality $Q(f1_{C_i},u)=Q(f,u)$ holds and, hence,  
 %\[Q_{C_i}(f1_{C_i},u)=Q(f1_{C_i},u)=Q(f,u)=\lambda_n\langle f,u\rangle_{L^2(X,m)}=\lambda_n\langle f1_{C_i},u\rangle_{L^2(C_i,m)}.\]  Therefore, $f1_{C_i}$ is an eigenfuction for the restricted 
 %Dirichlet form. 
 Hence, we can compute, using that $Q(f1_{C_i},f1_{C_j})=0,i\not=j$ by locality, 
 \begin{align*}
 Q(v)&=\sum_{i=1}^n c_i^2 Q(f1_{C_i})+2\sum_{i,j=1,i\not=j}^n c_i c_j Q(f1_{C_i},f1_{C_j})\\&=\sum_{i=1}^n c_i^2 Q(f1_{C_i})
 =\lambda_n \sum_{i=1}^n c_i^2 \|f1_{C_i}\|_{L^2(X,m)}^2\\&=\lambda_n\|v\|_{L^2(X,m)}^2=\lambda_n.       \end{align*}
Thus, we can apply Lemma \ref{lemma:closed:varia_eigen} und infer that $v$ is an eigenfunction for $\lambda_n$. But $v$ vanishes on $C_{n+1}$, a Dirichlet nodal domain, and hence $ v $ which contradicts (UC1) or (UC2).
\end{proof}
\begin{corollary}[Strong Courant bound]\label{corollary:strong_bound}
 Let $Q$ be an irreducible local, regular Dirichlet form with compact resolvent that satisfies $ \mathrm{(UC1)} $ or $ \mathrm{(UC2)} $.
 %. Suppose that $Q$ satisfies  (UC1), or (UC2) and every Dirichlet nodal domain   contains an open subset of $X$ with positive measure.
% Let $\lambda_1,\lambda_2,\ldots$ be the eigenvalues of $Q$ in ascending order counted with multiplicity. Let $f$ be an eigenfunction for $\lambda_n$. Assume that $f$ is continuous.
 Then, every eigenfunction for the eigenvalue $\lambda_n$ has at most $n$ topological nodal domains. 
\end{corollary}
\begin{proof}
 This follows from the previous theorem and since every topological nodal domain contains a Dirichlet nodal domain by Proposition~\ref{proposition:local:connected_comp_decomp_in_nodal_domains}.
\end{proof}

\subsubsection{Strongly elliptic differential operators}\label{section:elliptic}
Let $\Omega\subseteq \IR^n$ be a bounded, connected, open set.
On $C_c^\infty(\Omega)$ consider the operator given by 
$$L_0f:=\sum_{i,j\leq n} \nabla\cdot(a_{ij}\nabla)f+Vf,$$ where the real valued matrix $a(x)=(a_{ij}(x))_{i,j=1}^n$ is symmetric and $V:\Omega\to[0,\infty)$. Then, the Friedrich's extension $L$ of $L_0$ is the generator of a local, regular Dirichlet form $Q$. We assume the following three conditions which are satisfied under very mild assumptions on $ a $, $ V $ and the boundary of $ \Omega $:
\begin{itemize}
	\item $ D(Q)=H_0^1(\Omega). $
	\item $ Q $ has compact resolvent.
	\item The eigenfunctions of $ L $ are continuous.
\end{itemize}
For example this is satisfied if $ \Omega $ has smooth boundary, the coefficients satisfy 
$$\mu_1\|\xi\|\leq\sum_{i,j\leq n}a_{ij}(x)\xi_1\xi_j\leq \mu_2\|\xi\|$$ for some $\mu_1,\mu_2>0$ and  every $x\in\Omega$ and $\xi\in\IR^n$ and $ V $ is (sufficiently) bounded, see e.g. \cite[Lemma~6.1.3]{DaviesSpectral} and  \cite[Corollary~8.36.]{GilbargTrudinger}.
\begin{theorem}
 Under the assumptions stated above the topological nodal domains and the Dirichlet nodal domains of an eigenfunction  coincide (up to sets of measure zero). 
\end{theorem}
\begin{proof}
 For every eigenfunction $f$ the sets $F^+$ and $F^-$ are open. Therefore, every topological connected component of $F^+$ or $F^-$ is open in $\Omega$.  We show that these connected components coincide with the Dirichlet nodal domains. It suffices to consider $F^+$.
 First, we infer that $Q^{F^+}$ has domain $H_0^1(F^+)$ by definition of $D(Q^{F^+})$. Let $C$ be a topological connected components of $F^+$.  Lemma~\ref{lemma:local:equality_of_restrictions}
 yields that the restriction $(Q^{F^+})_C$ has domain $H_0^1(C)(=D(Q^C))$ and, thus, is a regular Dirichlet form. 
 Of course, $(Q^{F^+})_C$ is local as well. Hence, a set $B\subseteq C$ is $(Q^{F^+})_C$-invariant if and only if $1_Bf\in H_0^1(C)$ for every $f\in H_0^1(C)$. 
To finish the proof it suffices to show that there is no non-trivial $(Q^{F^+})_C$-invariant set $B$.

Assume there is a non-trivial $(Q^{F^+})_C$-invariant set $B$.
Note, that $H_0^1(C)$ is the domain of the classical Dirichlet form $\mathcal{E}^{(D)}$ in correspondence with the Dirichlet Laplacian $\Delta^{(D)}$ on $C$.
Take the first eigenvalue of $\Delta^{(D)}$, which is known to be simple, 
since $C$ is connected, and a corresponding eigenfunction $g$. Then the function $g1_B$ is an eigenfunction of $\Delta^{(D)}$ for the first eigenvalue, as well, since
for every $u\in H_0^1(C)$ one has $$\mathcal{E}^{(D)}(g1_B,u)=\mathcal{E}^{(D)}(g,1_Bu)=\lambda_1\langle g,u\rangle_{L^2(C,m)}.$$ 
This is  a contradiction to the simplicity of the eigenvalue.       
\end{proof}
\begin{corollary}
Under the assumptions stated above every eigenfunction $ f $ for $\lambda_n$ with multiplicity $ k $ has at most $(n+k-1)$ topological nodal domains.  Moreover, if  $ \mathrm{(UC1)} $ or $ \mathrm{(UC2)} $ holds then $ f $  has at most $n$ topological nodal domains. 
%Moreover,  if $Q$ satisfies either (UC1) or (UC2),  then $f$ has at most $n$ topological nodal domains. 
\end{corollary}
%\begin{remarks} Note, that in the case that (UC2) holds we do not need to assume that every Dirichlet nodal domain contains an open subset of positive measure, as the Dirichlet nodal domains coincide with the topological nodal domains, which are open by definition.  \end{remarks}
In particular, the corollary above can be applied to the Laplacian with Dirichlet boundary conditions. In this case, it is known, c.f. \cite{JerisonKenig}, 
that eigenfunctions do not vanish on non-empty open sets. Furthermore the Dirichlet nodal domains coincide with the topological nodal domains and are therefore open. Hence (UC2) holds in this case.

\subsubsection{Forms with doubly Feller resolvent}\label{section:Feller}
Next, we want to discuss a class of local, regular Dirichlet forms for which the concepts of topological nodal domains and Dirichlet nodal domains coincide. For this, we introduce the notion of doubly Feller 
resolvents. There is a vast amount of literature on notions of Feller properties for resolvents and semigroups. Unfortunately, these notions are not always consistent. Here, we follow \cite{FellerResolv}. 

Denote $C_0(X) $  be the closure of ${C_c(X)}$  with ${\|\cdot\|_\infty}$. Let  $B_b(X)$ be the space of 
bounded, real valued, Borel-measurable functions and $C_b(X)$ its subspace of continuous functions. A family of linear operators $$R_\alpha:B_b(X)\to B_b(X),\quad \alpha>0,$$ is called \emph{doubly Feller resolvent}, if 
\begin{itemize}
  \item [(1)] $R_\alpha C_0(X) \subseteq C_0(X)$ for every $\alpha>0$,
  \item[(2)] $R_\alpha B_b(X)\subseteq C_b(X)$ for every $\alpha>0$,
  \item[(3)] $\lim_{\alpha\to\infty} \alpha R_\alpha f(x)=f(x)$ for every $x\in X$, $f\in C_0(X)$
  \item[(4)] $R_\alpha f(x)-R_\beta f(x)=(\beta-\alpha)R_\alpha R_\beta f(x)$ for every $x\in X$, $f\in B_b(X)$
 \end{itemize}
hold. If $R_\alpha$, $\alpha>0$, satisfies only assertions (2) and (4) from above, then we call $R_\alpha$, $\alpha>0$, a \emph{strong Feller resolvent.} 

Let $Q$ be a Dirichlet form.
 It is well known that the resolvent $G_\alpha$, $\alpha>0$, of $Q$ can be extended to $L^\infty(X)$, see \cite[Theorem 1.4.1]{DaviesHeatKernels}. We denote this extension again by $G_\alpha$. Let $f\in B_b(X)$ and denote $[f]$ the equivalence class of all $m$-versions of $f$. Then, $[f]\in L^\infty(X)$ and we can define $G_\alpha f:= G_\alpha [f]$. 
\begin{definition}
 Let $Q$ be a Dirichlet form. 
 We say that the resolvent $G_\alpha$, $\alpha>0$,  has the \emph{strong Feller property} (respectively \emph{ the doubly Feller property}) if there is a strong Feller resolvent (respectively a doubly Feller resolvent) $R_\alpha$, $\alpha>0$, such that for every $u\in B_b(X)$ and $\alpha>0$ the function $R_\alpha u$ is an $m$-version of $G_\alpha u$.
\end{definition}
The next lemma shows that under the strong Feller property connectedness implies irreducibility.
\begin{lemma}[\cite{FellerResolv}, Lemma 4.1]\label{lemma:local:connected_implies_irreducible}
Every regular Dirichlet form on a connected set $ X $ whose resolvent has the strong Feller property is irreducible. 
\end{lemma}
The following result yields that the part of a form with doubly Feller resolvent on an open subset has a strong Feller resolvent.
\begin{lemma}[\cite{FellerResolv}, Theorem 3.1]\label{lemma:local:restr_strong_feller}
 Let $Q$ be a regular Dirichlet form whose resolvent is doubly Feller. Then, the resolvent of $Q^A$ has the strong Feller property for every open $ A\subseteq X $. 
\end{lemma}
Finally we can state the desired theorem.
\begin{theorem}
 Let $Q$ be a local, regular Dirichlet form on a locally connected space  whose resolvent is compact and  doubly Feller.   Then, the topological nodal domains of a continuous eigenfunction 
 coincide with the Dirichlet nodal domains (up to $m$-null sets). 
\end{theorem}
\begin{proof}
We only discuss $F^+$. By Proposition \ref{proposition:local:connected_comp_decomp_in_nodal_domains} we already know that a topological nodal domain $A$ can be decomposed 
into Dirichlet nodal domains. It is left to show that 
 $(Q^{F^+})_A$ is irreducible, then the result follows. By continuity of $f$ the set $F^+$ is open. 
 Since the space is locally connected, 
 we deduce that $A$ is open in $F^+$ and, thus, in $X$. Moreover, by Lemma~\ref{lemma:local:equality_of_restrictions} we infer $D(Q^A)=D((Q^{F^+})_A)$. Hence, it suffices to show that $Q^A$ is irreducible.
 By Lemma \ref{lemma:local:restr_strong_feller}, the resolvent of $Q^A$ has the strong Feller property. Since $A$ is a connected component, it is in particular connected. 
 Therefore, we can 
 apply Lemma \ref{lemma:local:connected_implies_irreducible} and the result follows.
\end{proof}
\begin{corollary} Let $Q$ be a local, regular Dirichlet form on a locally connected space  whose resolvent is compact and  doubly Feller. Then every continuous eigenfunction $ f $ of the eigenvalue $ \lambda_{n} $ with multiplicity $ k $ has at most $(n+k-1)$ topological nodal domains.
 Moreover,  if $Q$ satisfies either $ \mathrm{(UC1)} $ or $ \mathrm{(UC2)} $, 
 then $f$ has at most $n$ topological nodal domains. 
\end{corollary}
%\begin{remarks}
%Note, that in the case that (UC2) holds we do not need to assume that every Dirichlet nodal domain contains an open subset of positive measure, as the Dirichlet nodal domains coincide with the topological nodal domains, which are open by definition. 
%\end{remarks}
%Hence, for the class of forms discussed in the previous theorem the concept of $(Q,\mathcal{A})$-nodal domains coincides with the topological one. 
%
%\begin{corollary}
%Let $X$ be a locally compact, separable, locally connected metric space and let $m$ be a Radon measure of full support on $\mathcal{B}(X)$. Let $Q$ be a local, regular Dirichlet form
% with compact resolvent. Assume that the resolvent is doubly Feller.  Let $f$ be an eigenfunction of $Q$. Suppose that either (1)  $Q$ satisfies the unique continuation principle for sets of positive measure, or (2) $Q$ satisfies the unique continuation principle for open sets and every $(Q,\mathcal{A})$-nodal domain contains an open subset of $X$, 
% holds.
% Let $\lambda_1,\lambda_2,\ldots$ be the eigenvalues of $Q$ ordered by size. Let $f$ be an eigenfunction for $\lambda_n$. Then, $f$ has at most $n$ topological nodal domains. 
%\end{corollary}
\appendix
\section{Forms and forms in the wide sense}\label{section:forms} 
\subsection{Closed forms in the wide sense}
Let $Q$ be a real valued semi-scalar product defined on a subspace of $\ltwo$. Following the notation in \cite{Fuku}, $Q$ will be called \emph{form in the wide sense}. 
A densely defined form in the wide sense is called 
\emph{form}.

A form in the wide sense $Q$ on a subspace $D(Q)$ of $\ltwo$ is called \emph{closed}  if $D(Q)$ is a Hilbert space with respect to
$$\langle\cdot,\cdot\rangle_{Q,\alpha}:=Q(\cdot,\cdot)+\alpha\langle\cdot,\cdot\rangle_\ltwo$$ for one (and, thus, all) $\alpha>0$. We denote the corresponding norm by $\|\cdot\|_{Q_\alpha}$ and  set $\|\cdot\|_{Q}:=\|\cdot\|_{Q,1}$ and $\langle\cdot,\cdot\rangle_{Q}:=\langle\cdot,\cdot\rangle_{Q,1}$. Given a closed  form $Q$, there is a unique non-negative self-adjoint operator $L$ on $\ltwo$ with 
\begin{align*}
D(L):=\{u\in \ltwo\colon &\text{There is } v\in\ltwo \text{ such that } \\&Q(u,w)=\langle v,w \rangle_\ltwo \text{ for every } w\in D(Q)\},
\end{align*}
 \[Lu=v.\] 

The operator $L$ is called the \emph{generator} of the form $Q$. Of course, this generator gives rise to a resolvent. This yields that for every closed form there
is a resolvent in correspondence with the form. 

In the case of not densely defined closed forms in the wide sense, there is no unique generator in correspondence with the form. Nevertheless, one has still a resolvent, as the next lemma shows.
It is taken from \cite[Theorem 1.3.2]{Fuku}.
\begin{lemma}\label{FormReso}
Let $Q$ be a closed form in the wide sense. Then, for every $\alpha>0$ and $u\in\ltwo$ there is a unique  $G_\alpha u\in D(Q)$   that satisfies 
\[Q(G_\alpha u, v)+\alpha \langle G_\alpha u,v\rangle_\ltwo =\langle u,v\rangle_\ltwo\] for every $v\in D(Q)$ and the operators $G_\alpha:\ltwo\to\ltwo$ are linear, bounded and
satisfy the resolvent identity.

Moreover, for every $f\in \ltwo$ the function $\alpha \mapsto \alpha\langle f-\alpha G_\alpha f,f\rangle_\ltwo$ is non-negative and  monotonically increasing
\[f\in D(Q)\quad\Longleftrightarrow\quad \lim\limits_{\alpha\to\infty} \alpha\langle f-\alpha G_\alpha f,f\rangle_\ltwo<\infty.\]
 
Furthermore, for every $f,g\in D(Q)$, the equality 
\[Q(f,g)=\lim_{\alpha\to\infty} \alpha\langle f-\alpha G_\alpha f,g\rangle_\ltwo\] holds.
If $Q$ is densely defined, then the resolvent $ G_{\alpha} $
is strongly continuous and the equality $G_\alpha=(L+\alpha)^{-1}, \alpha>0$, holds.
\end{lemma}
Next we introduce eigenvalues and eigenvectors of forms in the wide sense. Since in general there is no generator for forms in the wide sense  we define them via the form.

Let $Q$ be a closed form in the wide sense. Then, $\lambda\in\IR$ is called \emph{eigenvalue} of $Q$ if   
there is an $u\in D(Q)$ that satisfies \[Q(u,v)=\lambda\langle u,v\rangle_\ltwo\] for every $v\in D(Q)$. Such an $u$ is then called \emph{eigenvector} for the form $Q$
(and the eigenvalue $\lambda$). 
Note that the eigenvalues of $Q$ are non-negative, since for an eigenfunction $f$ with eigenvalue $\lambda$ one has $$0\leq Q(f)=\lambda\langle f,f\rangle_{L^2(X,m)}.$$

The following lemma shows the connection of the eigenvalues of a closed form in the wide sense $Q$ and the eigenvalues 
of the resolvent and, if $Q$ is densely defined,  of the eigenvalues of the generator. 
\begin{lemma}\label{lemma:closed:eigenvalue_resolvent}
Let $Q$ be a closed form in the wide sense and $G_\alpha$, $\alpha>0$, the resolvent of $Q$. Let $\lambda\in\IR$. Then, the following are equivalent:
\begin{itemize}
\item[(i)] $\lambda$ is an eigenvalue of $Q$, 
\item[(ii)] $\frac{1}{\lambda+\alpha}$ is an eigenvalue of $G_\alpha$ for every $\alpha>0$.
\end{itemize}
 Moreover, if $Q$ is densely defined, and, hence, has a generator $L$, then both assertions are equivalent to
 \begin{itemize}
\item[(iii)] $\lambda$ is an eigenvalue of $L$. 
\end{itemize}   
\end{lemma} 
\begin{proof}
(i) $ \Longrightarrow $ (ii): Let $f\in D(Q)$ be an eigenvector of $Q$ with eigenvalue $\lambda$. Let $\alpha>0$ be arbitrary. Then, for every $g\in D(Q)$ we have 
$$\langle f,g\rangle_{Q,\alpha}=(\alpha+\lambda)\langle f,g\rangle_{L^2(X,m)}=(\alpha+\lambda)\langle G_\alpha f,g\rangle_{Q,\alpha},$$ where the last equality follows by the definition of the function $G_\alpha f$, see the first part of Lemma~\ref{FormReso}. 
Since $g\in D(Q)$ was arbitrary, and since $(D(Q),\langle\cdot,\cdot\rangle_{Q,\alpha})$ is a Hilbert space, we infer $G_\alpha f =\frac{1}{\lambda+\alpha}f$ and (ii) follows.
 
(ii) $ \Longrightarrow $ (i):  We first show  $ G_{\alpha} $ has the same the eigenvectors for all $ \alpha $. Let 
 $f$ be an eigenvector of $G_\alpha$ for eigenvalue $\mu$ and let $\lambda$ be such that $\mu=\frac{1}{\lambda+\alpha}$ and let $\beta\not=\alpha$, $\beta>0$. Then, using the resolvent identity we infer 
 $$G_\beta f=(\alpha-\beta)G_\beta G_\alpha f+G_\alpha f=(\alpha-\beta)\frac{1}{\lambda+\alpha} G_\beta f+\frac{1}{\lambda+\alpha} f.$$ Therefore, $f$ is an eigenvector of $G_\beta$ for 
 eigenvalue $\frac{1}{\lambda+\beta}$.
 
 Let $\frac{1}{\lambda+\alpha}$ be an eigenvalue of $G_\alpha$ for every $\alpha>0$ and let $f\in L^2(X,m)$ be an eigenfunction. Then, using Lemma~\ref{FormReso} we infer 
 $\frac{1}{\lambda+\alpha}f=G_\alpha f\in D(Q)$
 and, thus, $f\in D(Q)$. Again using Lemma~\ref{FormReso} we deduce (i) by
 \begin{align*}
 Q(f,g)&=\lim\limits_{\alpha\to\infty} \alpha\langle f-\alpha G_\alpha f,g\rangle_\ltwo=\lim\limits_{\alpha\to\infty} \alpha\left\langle f-\frac{\alpha}{\lambda+\alpha}  f,g\right\rangle_\ltwo\\
 &=\lim\limits_{\alpha\to\infty}  \frac{\alpha}{\lambda+\alpha} \lambda\left\langle f,g\right\rangle_\ltwo=\lambda \langle  f,g\rangle_\ltwo.
 \end{align*}
The equivalence (i) $ \Longleftrightarrow $ (iii) is standard.  %Let $Q$ be densely defined. Suppose $\lambda$ is an eigenvalue of $L$. Then, there is an $f\in D(L)\subseteq D(Q)$ such that $Lf=\lambda f$ holds. Since $L$ is the generator of $Q$, for every $u\in D(Q)$ one has  \[Q(f,u)=\langle Lf,u\rangle=\lambda\langle f, u\rangle.\] Thus, $\lambda$ is an eigenvalue for $Q$ and $f$ is an eigenvector for $Q$.
%
%For the remaining implication suppose that there is $f\in D(Q)$ such that  $Q(f,u)=\lambda\langle f, u\rangle$  holds for every $u\in D(Q)$. 
%Then, one has $f\in D(L)$ and $Lf=\lambda f$ by definition of the generator $L$. This concludes the proof. 
\end{proof}

A closed form in the wide sense is said to have \emph{compact resolvent} if one (and, thus, all) of the operators  $G_\alpha$, $\alpha>0$, is compact. 

\begin{lemma}\label{CompAeqComp}
Let $Q$ be a closed form in the wide sense on $\ltwo$ with resolvent $G_\alpha,\alpha>0$. Then, $Q$ has compact resolvent  (i.e. for every $\alpha>0$ 
the operator $G_\alpha:\ltwo\to\ltwo$ is compact) if and only if for every $\alpha>0$ the operator $G_\alpha:\ltwo\to (D(Q),\|\cdot\|_Q)$ is compact.
\end{lemma}
\begin{proof}
Let $\alpha>0$ be arbitrary.
One direction is clear, since the inequality $\|f\|_Q\geq \|f\|_\ltwo$ holds for every $f\in D(Q)$.
 For the other direction let $(f_n)$ be a bounded sequence in $L^2(X,m)$. Using the compactness, $(G_\alpha f_n)$ is (without loss of generality) a Cauchy sequence in $\ltwo$. By  Lemma~\ref{FormReso}
\[\|G_\alpha (f_n-f_m)\|_{Q,\alpha}=\langle G_\alpha (f_n-f_m),f_n-f_m\rangle_\ltwo \to 0,\] as $f_n-f_m$ is bounded in $L^2(X,m)$.   
\end{proof}
\begin{remarks}
It is obvious that if a closed form has compact resolvent, then the corresponding operator has compact resolvent and, hence, discrete spectrum. Since $L^2(X,m)$ is
separable by assumption, we infer that if a closed form has compact resolvent, then there is an orthonormal basis of $L^2(X,m)$ consisting of eigenfunctions of $L$ and, therefore, of $Q$. 
\end{remarks}
Next we prove the following well-known, c.f. \cite{reed_simon_IV}, variational principle for eigenfunctions of closed forms with compact resolvent. 
\begin{lemma}\label{lemma:closed:varia_eigen}
 Let $Q$ be a closed form in the wide sense with compact resolvent
  Let $v\in D(Q)$ be $L^2$-orthogonal to the eigenfunctions  of the first $ n-1 $ eigenvalues   of  
  $Q$. Then,  \[Q(v)\geq \lambda_n\|v\|_{L^2(X,m)}\] holds for the $ n $-th eigenvalue $ \lambda_{n} $ and if equality holds, then $v$ is an eigenfunction of $Q$ for $\lambda_n$. 
\end{lemma}
\begin{proof}
 For the first part we use that $L^2(X,m)$ is separable and, hence, the eigenvectors form an orthonormal basis of $L^2(X,m)$. We expand $v$ into eigenfunctions and infer
 \[v=\sum_{k\geq n} \langle v, f_k\rangle_{L^2(X,m)} f_k.\] 
 Using this, a direct calclation yields
\begin{align*}
 Q\left(v-\sum_{k=n}^N \langle v, f_k\rangle_{L^2(X,m)} f_k\right)
=Q(v)-\sum_{k=n}^N \lambda_k\langle v, f_k\rangle_{L^2(X,m)}^2.
 \end{align*}
Since the left hand side is nonnegative, we infer $$Q(v)\geq\sum_{k=n}^N \lambda_k\langle v, f_k\rangle_{L^2(X,m)}^2.$$
%\begin{align*}
% 0\leq& Q\left(v-\sum_{k=n}^N \langle v, f_k\rangle_{L^2(X,m)} f_k\right)\\
% =&Q(v)-2\sum_{k=n}^N \langle v, f_k\rangle_{L^2(X,m)} Q(v,f_k)\\&+\sum_{k,l=n}^N \langle v, f_k\rangle_{L^2(X,m)}\langle v, f_l\rangle_{L^2(X,m)} Q(f_k,f_l)\\
% =&Q(v)-2\sum_{k=n}^N \lambda_k(\langle v, f_k\rangle_{L^2(X,m)})^2\\&+\sum_{k,l=n}^N \langle v, f_k\rangle_{L^2(X,m)}\langle v, f_l\rangle_{L^2(X,m)} \lambda_k\langle f_k,f_l\rangle_{L^2(X,m)}\\
% =&Q(v)-2\sum_{k=n}^N \lambda_k(\langle v, f_k\rangle_{L^2(X,m)})^2+\sum_{k=n}^N\lambda_k \langle (v, f_k\rangle_{L^2(X,m)})^2\\
%=&Q(v)-\sum_{k=n}^N \lambda_k(\langle v, f_k\rangle_{L^2(X,m)})^2.
% \end{align*}
 Now ,taking the limit $N\to\infty$ and using Parseval's identity yields \[Q(v)\geq \sum_{k=n}^\infty \lambda_k\langle v, f_k\rangle_{L^2(X,m)}^2\geq\lambda_n \sum_{k=n}^\infty\langle v, f_k\rangle_{L^2(X,m)}^2=\lambda \|v\|_{L^2(X,m)}^2.\]  
 For the second part, the assumption and the identity above  implies 
 \[\sum_{k=n}^\infty \lambda_k\langle v, f_k\rangle_{L^2(X,m)}^2=\lambda_n \sum_{k=n}^\infty\langle v, f_k\rangle_{L^2(X,m)}^2.\] This 
 yields that \[\langle v, f_k\rangle_{L^2(X,m)}=0\] holds for all $f_k$ that are not eigenfunctions for $\lambda_n$. Since there are only finitely many eigenfunctions for $\lambda_n$,
 the function $v$ is a  linear combination of eigenfunctions for $\lambda_n$ and, therefore, 
 an eigenfunction for $\lambda_n$. 
\end{proof}

\subsection{Restrictions of forms in the wide sense}\label{section:parts_of_forms}
Next, we want to discuss restrictions of closed forms in the wide sense to measurable sets and with respect to\emph{ nests}, i.e., increasing sequences of measurable subsets 
of $X$.  
\begin{definition}
Let $Q$ be a form in the wide sense on $\ltwo$. For  a measurable set $A\subseteq X$,  the \emph{restriction}  $Q_A:=Q|_{D(Q_A)}$ \emph{of} $Q$ \emph{to} $A$ is given by  \[D(Q_A):=\{u\in D(Q): u1_{X\setminus A}=0\}.\] 
For a nest $\mathcal{A}:=(A_n)_{n\in\IN}$  the \emph{restriction $Q_A^{\mathcal{A}}$ of $Q$ to $A$ with respect to $\mathcal{A}$} has
$$D(Q_A^{\mathcal{A}})=\overline{\bigcup_{n=1}^\infty D(Q_{A\cap A_n})},$$ as domain, where the closure is taken in $(D(Q),\|\cdot\|_Q)$. 
\end{definition}
\begin{remarks}
 Let $Q$ be a  form in the wide sense on $\ltwo$ and  $A\subseteq B\subseteq X$ be measurable. Then, 
and, hence,  $Q_A=(Q_B)_A$.
\end{remarks}
\begin{remarks}
 Let $Q$ be a  form in the wide sense on $\ltwo$ and let $A$ be measurable. For a nest $\mathcal{A}$ with $A_n=X$ for every $n$ we obviously have $Q_A^{\mathcal{A}}=Q_A$. Therefore, the concept of restriction with respect to a nest is a generalization of the concept of restriction to a set. 
\end{remarks}
\begin{remarks}
 Let $Q$ be a  form in the wide sense on $\ltwo$,let $A$ be measurable and $\mathcal{A}$ be a nest. Since $\|\cdot\|_Q$-convergence implies $L^2(X,m)$-convergence, one has $D(Q_A^{\mathcal{A}})\subseteq D(Q_A).$ 
\end{remarks}
\begin{lemma}\label{ResClos}
Let $Q$ be a closed form in the wide sense, let $A\subseteq X$ be measurable and let $\mathcal{A}:=(A_n)_{n\in\IN}$ be a nest. 
Then, $Q_A$ and  $Q_A^{\mathcal{A}}$ are closed forms in the wide sense.
\end{lemma}
\begin{proof}
Observe that $\|\cdot\|_{Q_A}$ is the restriction of $\|\cdot\|_Q$ to $D(Q_A)$. Let $(f_n)$ be a Cauchy sequence in $(D(Q_A),\|\cdot\|_{Q_A})$ and, thus, a Cauchy sequence in 
$(D(Q),\|\cdot\|_{Q})$. Therefore, there is $f\in D(Q)$ such that $\|f_n-f\|_Q\to 0$. But $\|\cdot\|_Q$-convergence implies $\ltwo$-convergence. Hence, $f|_{X\setminus A}=0$ almost everywhere. This yields $f\in D(Q_A)$. For a nest $\mathcal{A}:=(A_n)_{n\in\IN}$ the domain $D(Q_{A}^{\mathcal{A}})$ is obviously a closed subspace of $(D(Q),\|\cdot\|_Q)$. Hence, $Q_{A}^{\mathcal{A}}$ is a closed form in the wide sense. 
\end{proof}
By the previous lemma, the restrictions $Q_A$ and $ Q_{A}^{\mathcal{A}} $ of a closed form $Q$ in the wide sense are closed form in the wide sense and, thus, 
have a resolvent, c.f. Lemma~\ref{FormReso}, which we denote by $G_\alpha^A$ and  $G_\alpha^{\mathcal{A},A}$, $\alpha>0$. 

The following proposition shows the connection between the resolvent and the resolvent of the restriction. The first part can be extracted from the proof of  \cite[Proposition~2.10]{GreenFor}.

\begin{proposition}\label{ProjectRes}
Let $Q$ be a closed form in the wide sense on $\ltwo$, let $A\subseteq X$ be measurable and  $\mathcal{A}$ be a nest. %Then, the space $D(Q_A)$ and $D(Q_{A}^{\mathcal{A}})$ are closed subspaces of ${(D(Q), \|\cdot\|_Q)}$. 
For the  orthogonal projections $P_A$ and $P_{\mathcal{A},A}$  from $(D(Q),\|\cdot\|_{Q})$ to $(D(Q_A),\|\cdot\|_{Q_{A}})$ and $(D(Q_A^{\mathcal{A}}),\|\cdot\|_{Q_{A}^{\mathcal{A}}})$
we have\[G_1^A =P_AG_1, \quad\mbox{and}\quad G_1^{\mathcal{A},A} =P_{\mathcal{A},A} G_1\qquad \mbox{on }\ltwo.\] 
\end{proposition}
\begin{proof}
Obviously the space  $D(Q_A)$ is closed in $(D(Q), \|\cdot\|_Q)$, since $Q_A$ is a closed form in the wide sense and $\|\cdot\|_{Q_A}$ is the restriction of $\|\cdot\|_Q$ to $D(Q_A)$.
By Lemma~\ref{FormReso} we infer 
 for every $f\in L^2(A,m)$ and
$g\in D(Q_A)$
\[\langle G_1 f,g\rangle_{Q}=\langle f,g\rangle_\ltwo=\langle f,g\rangle_{L^2(A,m)}=\langle G_1^A f,g\rangle_{Q_A},\]
and, furthermore,
\[\langle G_1 f,g\rangle_{Q}=\langle P_AG_1 f,g\rangle_{Q}+\langle (I-P_A)G_1 f,g\rangle_{Q}=\langle P_AG_1 f,g\rangle_{Q}=\langle P_AG_1 f,g\rangle_{Q_A}\] holds as well. 
Hence, we deduce $G_1^A f=P_AG_1 f$ for every $f\in L^2(A,m)$.

Now let $f\in\ltwo$ be arbitrary. Then, we have 
\[G_1^A f=G_1^A (f1_A)=P_AG_1 (f1_A)=P_AG_1 f-P_AG_1 (f1_{X\setminus A}).\] It is left to show $P_AG_1 (f1_{X\setminus A})=0$. Since this function is an element of $D(Q_A)$ it suffices to show $\langle P_AG_1 (f1_{X\setminus A}),g\rangle_{Q_A}=0$ for every $g\in D(Q_A)$. This can be seen by the following calculation, using the self-adjointness of the projection $P_A$ on $D(Q)$,
\begin{align*}
\langle P_AG_1 (f1_{X\setminus A}),g\rangle_{Q_A}
&=\langle P_AG_1 (f1_{X\setminus A}),g\rangle_{Q}=\langle G_1 (f1_{X\setminus A}),P_Ag\rangle_{Q}\\
&=\langle G_1 (f1_{X\setminus A}),g\rangle_{Q}=\langle f1_{X\setminus A},g\rangle_\ltwo=0.
\end{align*}
This shows that $Q_A$ is a closed form in the wide sense on $L^2(A,m)$.

The proof of the equality $G_1^{\mathcal{A},A} =P_{\mathcal{A},A} G_1$ works analogously.
\end{proof}

Next, we show that compactness of the resolvent carries over to  restrictions.
\begin{proposition}\label{RestrCpt}
Let $Q$ be a closed form in the wide sense  on $\ltwo$ with compact resolvent and let $A\subseteq X$ be measurable and $\mathcal{A}$ be a nest. Then, $Q_A$ and $Q_A^{\mathcal{A}}$ have compact resolvent.
\end{proposition}
\begin{proof}
This follows immediately from the  $G_1^A=P_AG_1$ and $G_1^{\mathcal{A},A}=P_{\mathcal{A},A}G_1$ as  products of  bounded with  compact (by Lemma~\ref{CompAeqComp}) operators.
\end{proof} 

\subsection{Positivity preserving forms}  
A closed form in the wide sense $Q$ on a subspace $D(Q)$ of $\ltwo$ is called \emph{positivity preserving} if for every $u\in D(Q)$ one has $|u|\in D(Q)$ and $Q(|u|)\leq Q(u)$.  

\begin{lemma}\label{PosPartIneq}
Let $Q$ be a positivity preserving form in the wide sense. Then, for every $u\in D(Q)$ one has $u_+,u_-\in D(Q)$ and $$ Q(u_+)\leq Q(u),\quad Q(u_-)\leq Q(u)\quad\mbox{ and }\quad Q(u_+,u_-)\leq 0. $$
\end{lemma}
\begin{proof}
 We get $u_+,u_-\in D(Q)$ by $u_+=\frac12(u+|u|)$ and $u_-=\frac12(u-|u|)$.
 The inequality $Q(u_+,u_-)\leq 0$ follows from \[Q(u_+)+2Q(u_+,u_-)+Q(u_+)=Q(|u|)\leq Q(u)=Q(u_+)-2Q(u_+,u_-)+Q(u_-).\]
  Furthermore,
\[Q(u)=Q(u_+)+Q(u_-)-2Q(u_+,u_-)\geq Q(u_+)\] by $Q(u_+,u_-)\leq 0$. The inequality for $u_-$ can be seen analogously.% This concludes the proof.
\end{proof}
The following lemma is proven in \cite[Theorem XIII.50]{reed_simon_IV} in the context of densely defined forms. We include it for convenience of the reader.
\begin{lemma}\label{ResPosPres}
Let $Q$ be a positivity preserving form in the wide sense. Then, the resolvent satisfies $G_\alpha f\geq 0$ for every $\alpha>0$ and every non-negative $f\in\ltwo$.
\end{lemma}
\begin{proof}
 Let $f\in\ltwo$ with $f\geq 0$ and $g\in D(Q)$, $g\geq 0$. Then, one has
 \begin{align*}
 \|G_\alpha f+g\|_{Q,\alpha}&=\|G_\alpha f\|_{Q,\alpha}+\|g\|_{Q,\alpha}+2\langle G_\alpha f,g\rangle_{Q,\alpha}\\
 &=\|G_\alpha f\|_{Q,\alpha}+\|g\|_{Q,\alpha}+2\langle f,g\rangle_\ltwo\\
% &\geq \|G_\alpha f\|_{Q,\alpha}+\|g\|_{Q,\alpha}\\
 &\geq \|G_\alpha f\|_{Q,\alpha}.
 \end{align*}
In particular we can apply this for $g=|G_\alpha f|-G_\alpha f$ and infer 
\[\||G_\alpha f|\|_{Q,\alpha}\geq \|G_\alpha f\|_{Q,\alpha}\] and, since $Q$ is positivity preserving, $\|G_\alpha f\|_{Q,\alpha}=\||G_\alpha f|\|_{Q,\alpha}$ follows. Hence, \[\|g\|_{Q,\alpha}+2\langle f,g\rangle_\ltwo=0\] and, as 
the second summand is non-negative, we deduce $\|g\|_{Q,\alpha}=0$. Since $\|\cdot\|_{Q,\alpha}$ is a norm we get $g=0$ and, hence, $|G_\alpha f|=G_\alpha f$. %This concludes the proof.
\end{proof}

\begin{lemma}\label{InvPosPres}
Let $Q$ be a positivity preserving form in the wide sense and let $A\subseteq X$ be measurable and $\mathcal{A}$ be a nest. Then, $Q_A$ and $Q_A^{\mathcal{A}}$ are positivity preserving forms in the wide sense on $L^2(A,m)$.
\end{lemma}
\begin{proof}
The form in the wide sense $Q_A$ is closed as seen in Lemma \ref{ResClos}.
Let $f\in D(Q_A)$ be arbitrary. Then $f\in D(Q)$ and, hence, $|f|\in D(Q)$. Furthermore, the equality  $|f|1_{X\setminus A}=|f1_{X\setminus A}|=0$ holds. 
This yields $|f|\in D(Q_A)$. Finally, we infer
\[Q_A(|f|)=Q(|f|)\leq Q(f)=Q_A(f).\] 

Let $\mathcal{A}$ be a nest.
Let $f\in D(Q_{A}^{\mathcal{A}})$ be arbitrary. We show $|f|\in D(Q_{A}^{\mathcal{A}})$.
 Let $(f_n)$ be a sequence in  $\bigcup_{n=1}^\infty D(Q_{A\cap A_n})$ such that $\|f_n-f\|_{Q_{A}^{\mathcal{A}}}\to 0,$ $n\to\infty$. 
 First, we infer $|f_n|\in \bigcup_{n=1}^\infty D(Q_{A\cap A_n})$ as every $Q_{A\cap A_n}$ is positivity preserving.
  Moreover, we have $\| |f_n|-|f|\|_{L^2(A,m)}\to 0$ 
  and $$\sup_{n\in\IN} \| |f_n|\|_{Q_{A}^{\mathcal{A}}}\leq \sup_{n\in\IN} \|f_n\|_{Q_{A}^{\mathcal{A}}}<\infty.$$ 
  Hence, the Banach-Saks theorem, see for example \cite[Theorem A.4.1]{ChenFuku} yields that the Ces\`aro-means of a subsequence of $(|f_n|)$ converge with respect to $\|\cdot\|_{Q_{A}^{\mathcal{A}}}$ to $|f|$. 
  Since these Ces\`aro means are elements of $\bigcup_{n=1}^\infty D(Q_{A\cap A_n})$ we infer $|f|\in D(Q_{A}^{\mathcal{A}})$.
  Finally,   $ Q_A^{\mathcal{A}}(|f|)=Q(|f|)\leq Q(f)=Q_A^{\mathcal{A}}(f). $% This concludes the proof. 
\end{proof}
The following proposition shows a useful inequality for the projection $P_A$. It can be extracted from the proof of  \cite[Proposition~2.10]{GreenFor}.
\begin{proposition}\label{ProjectIneq}
Let $Q$ be a positivity preserving form in the wide sense on $\ltwo$ and let $A\subseteq X$ be measurable and $\mathcal{A}$ be a nest.
Let $P_A$ and $P_{\mathcal{A},A}$  be the orthogonal projections from $D(Q)$ to $D(Q_A)$ and $D(Q_{A}^{\mathcal{A}})$. Then, for every $f\in D(Q)$ with $f\geq 0$ one has \[P_A f\leq f\quad\mbox{ and }\quad P_{\mathcal{A},A}f\leq f.\]
\end{proposition}
\begin{proof}
Let $f\geq 0$ in $D(Q)$. Then, $f\wedge P_Af\in D(Q)$, since $f\wedge P_Af=\frac12(f+P_Af-|f-P_Af|)$  and $Q$ is  positivity preserving. Using that  $f$ is non-negative and that $(P_A f)1_{X\setminus A}= 0$ we infer $f\wedge P_Af\in D(Q_A)$. Note, that $f-f\wedge P_A f=(f-P_A f)_+$ holds.

Moreover, we conclude
\begin{align*}
\|f-f\wedge P_Af\|_Q^2=\|(f-P_Af)_+\|_Q^2\leq \|f-P_Af\|_Q^2,
\end{align*}
since $\|u_+\|_Q\leq\|u\|_Q$ for every $u\in D(Q)$ by Lemma~\ref{PosPartIneq}. Hence, $f\wedge P_Af$ has a smaller distance to $f$ than $P_Af$. But $P_A f$ is the unique distance minimizing object in $D(Q_A)$ and, thus, $P_Af=f\wedge P_Af$ holds. This shows $P_Af\leq f$. 

To prove the second inequality, we show $u\wedge P_{\mathcal{A},A} u\in D(Q_{A}^{\mathcal{A}})$ firs
Let $(v_n)$ be a sequence in $\bigcup_{n=1}^\infty D(Q_{A\cap A_n})$ such that $\|v_n-P_{\mathcal{A},A} u\|_{Q_A^{\mathcal{A}}}\to 0,$ $n\to\infty$. 
Then, we get $v_k\wedge u\in \bigcup_{n=1}^\infty D(Q_{A\cap A_n})$ for every $k\in\IN$, as $u\geq 0$ and every $v_k$ vanishes outside of some set $A_l\cap A$.  
Moreover, we have $$v_n\wedge u\to P_{\mathcal{A},A} u\wedge u\quad\text{ in }L^2(X,m)$$ and 
$$\sup_{n\in\IN}\|u\wedge v_n\|_Q\leq\sup_{n\in\IN}(2\|u\|_Q+2\|v_n\|_Q)<\infty.$$ Using Banach-Saks $u\wedge P_{\mathcal{A},A} u\in D(Q_{A}^{\mathcal{A}})$ follows. Now, the proof of the inequality $P_{\mathcal{A},A}f\leq f$  follows using exactly the same arguments as above.%in the proof of the inequality $P_A f\leq f$. 
\end{proof}

\section{Invariant sets}\label{section:invariance}
In this section we introduce the notion of $Q$-invariance of a set. In the standard literature, c.f. \cite{Fuku}, this concept is introduced for densely defined forms. However, the proofs of the basic results can be carried over to our more general setting.  We use this notions in section~\ref{subsection:connected_components}
to define the concept of connected components for a positivity preserving form in the wide sense, which in turn is then used to define nodal domains in Section~\ref{section:nodal_domains}.

\begin{definition}\label{definition:invariance}
Let $Q$ be a closed form in the wide sense with resolvent $G_\alpha$, $\alpha>0$. A measurable set $A\subseteq X$ is called \emph{$Q$-invariant} if it is invariant for every 
$G_\alpha,\alpha>0$, i.e., if for every $f\in \ltwo$ one has $$G_\alpha(1_A f)=1_A G_\alpha f.$$ 
\end{definition}
%The next result shows that $Q$-invariance is equivalent to invariance with respect to a single operator $G_\alpha$ 
%\begin{proposition}\label{ResInv}(\cite{Schilling}, Theorem 27)
%Let $Q$ be a Dirichlet form with resolvent $(G_\alpha)$. Then, the following assertions are equivalent:
%\begin{itemize}
%\item[(i)] $A$ is invariant for every $G_\alpha$.
%\item[(ii)] $A$ is invariant for one $G_\alpha$.
%\end{itemize} 
%\end{proposition}
%\begin{proof}
%TODO
%\end{proof}
The next lemma shows that the $Q$-invariant sets form a $\sigma$-algebra. The lemma can be found in \cite[Lemma 1.6.1]{Fuku} for Dirichlet forms.%(see Definition~\ref{def:dirichlet_form}).
\begin{lemma}\label{lemma:Inv:Inv_is_sigma_field}
 Let $Q$ be a closed form in the wide sense. Then, the family of $Q$-invariant sets is a $\sigma$-algebra.
\end{lemma}
\begin{proof}
Obviously $\emptyset$ and $X$ are $Q$-invariant. %Next we prove that a set $A$ is $Q$-invariant if and only if $X\setminus A$ is $Q$-invariant. By symmetry, one only has to prove one implication. 
Let $A$ be invariant for every $G_\alpha,\alpha>0$. One has \[G_\alpha (1_Af)+G_\alpha (1_{X\setminus A}f)=G_\alpha f=1_A G_\alpha f+1_{X\setminus A} G_\alpha f.\] By $G_\alpha (1_A f)=1_A G_\alpha f$, the set $ X\setminus A $ is $ G_{\alpha} $ and, thus, $ Q $-invariant.

 Now let $(A_i)_{i\in \IN}$ be $Q$-invariant.
For, $g\in \ltwo$, $\alpha>0$, $ i,j $, we get
 \[G_\alpha (g1_{A_{i}\cap A_{j}})=G_\alpha (1_{A_{i}}(1_{A_{j}}g))=1_{A_{i}}G_\alpha (1_{A_{j}} g)=1_{A_{i}} 1_{A_{j}} G_\alpha g=1_{A_{i}\cap A_{j}}G_\alpha g.\] Thus, $A_{i}\cap A_{j}$ is $Q$-invariant. Moreover, \[g1_{\bigcap_{i=1}^\infty A_i} =\lim_{n\to\infty}g1_{\bigcap_{i=1}^n A_i}\text{ in }L^2(X,m)\] and,
by the continuity of $G_\alpha$ and the $Q$-invariance of $\bigcap_{i=1}^n A_i$, we deduce 
 \begin{align*}
                                                G_\alpha (g1_{\bigcap_{i=1}^\infty A_i}) =\lim_{n\to\infty}G_\alpha (g1_{\bigcap_{i=1}^n A_i})=\lim_{n\to\infty}1_{\bigcap_{i=1}^n A_i} G_\alpha g
                                                =1_{\bigcap_{i=1}^\infty A_i} G_\alpha g.
                                               \end{align*}
 This completes the proof.
\end{proof}

The next lemma gives a characterization of $Q$-invariance via  $Q$. Again, it can be found in \cite[Theorem 1.6.1]{Fuku} for Dirichlet forms.
\begin{lemma}\label{CharInv}
Let $Q$ be a positivity preserving form in the wide sense. A measurable set $A$ is $Q$-invariant if and only if 
$1_A f\in D(Q)$ and \[Q(f)=Q(1_A f)+Q(1_{X\setminus A}f),\qquad f\in D(Q).\]
\end{lemma}
\begin{proof}
Let $A$ be $Q$-invariant. Then, by Lemma~\ref{FormReso}, one has  for  $f\in D(Q)$
\begin{align*}
Q(f)=\lim_{\alpha\to\infty}&\alpha\big(\langle f1_A,f1_A\rangle_\ltwo+\langle f1_{X\setminus A},f1_{X\setminus A}\rangle_\ltwo\\
&-\langle  \alpha G_\alpha f1_{A}, f1_{A}\rangle_\ltwo-\langle  \alpha G_\alpha f1_{X\setminus A}, f1_{X\setminus A}\rangle_\ltwo\\
&-2\langle  \alpha G_\alpha f1_{ A}, f1_{X\setminus A}\rangle_\ltwo\big).
\end{align*}

By $Q$-invariance of $A$ one has $G_\alpha 1_A f=1_A G_\alpha f$. Thus, \[2\langle  \alpha G_\alpha f1_{ A}, f1_{X\setminus A}\rangle_\ltwo=0\] follows.  Since the function $\alpha\mapsto \alpha\langle u-G_\alpha u,u\rangle_\ltwo$ is monotonically increasing and non-negative for  $u\in\ltwo$,  by Lemma~\ref{FormReso}, both limits 
\begin{align*}
Q(f_{ A})&=\lim_{\alpha\to\infty}\alpha\langle f1_A-G_\alpha f1_A,f1_A\rangle_\ltwo\\
Q(f_{X\setminus A})&=\lim_{\alpha\to\infty}\alpha\langle f1_{X\setminus A}-G_\alpha f1_{X\setminus A},f1_{X\setminus A}\rangle_\ltwo
\end{align*}
  exist. Hence, the first direction is proven. 

For the other direction suppose assume $1_A f\in D(Q)$ and $$Q(f)=Q(1_A f)+Q(1_{X\setminus A}f)$$  for  all $f\in D(Q)$. By the polarization identity we infer that 
$$Q(f,g)=Q(1_A f,1_A g)+Q(1_{X\setminus A}f,1_{X\setminus A} g)$$ holds for every $f,g\in D(Q)$ and, hence,  one has
 \[Q(1_A f,g)=Q(1_{X\setminus A} 1_A f,1_{X\setminus A}g)+Q(1_A1_Af,1_Ag)=Q(1_Af,1_Ag).\] Interchanging the roles of $f$ and $g$ yields \[Q(1_Af,g)=Q(f,1_Ag).\]
By Lemma~\ref{FormReso}, we can calculate for  $f\in\ltwo$, $g\in D(Q)$ and $\alpha>0$ \begin{align*}
 \langle G_\alpha 1_A f,g\rangle_{Q,\alpha}
 &=\langle1_A f,g\rangle_\ltwo=\langle f,1_A g\rangle_\ltwo= \langle G_\alpha  f,1_Ag\rangle_{Q,\alpha}\\
 &= \langle 1_AG_\alpha f,g\rangle_{Q,\alpha}.
 \end{align*}
 Since $(D(Q),\langle\cdot,\cdot\rangle_{Q,\alpha})$ is a Hilbert space we infer $G_\alpha (1_Af)=1_AG_\alpha f$.
% This concludes the proof.
\end{proof}

We  know by Proposition~\ref{ProjectRes} that  restrictions of a closed form in the wide sense are closed.
Next, we show that if restricting  to an invariant set, then the resolvent of the restricted form is the restriction of the resolvent.
\begin{lemma}\label{RestrRes}
Let $Q$ be a closed form in the wide sense  and let $A$ be $Q$-invariant.  Then, the resolvent $G_\alpha^A$, $\alpha>0$, of $Q_A$ satisfies $G_\alpha^A=G_\alpha|_{L^2(A,m)}$ and if $Q$ is densely defined on $L^2(X,m)$, then $Q_A$ is densely defined on $L^2(A,m)$. 
\end{lemma}
\begin{proof}
 By $Q$-invariance, the resolvent $G_\alpha$  maps   $L^2(A,m)$ to $L^2(A,m)$.  Hence, $G_\alpha|_{L^2(A,m)}$ is the resolvent of $Q_A$ since for  $u\in L^2(A,m), v\in D(Q_A)$  \begin{align*}
  Q_A(G_\alpha u, v)+\alpha \langle G_\alpha u,v\rangle_{L^2(A,m)}
  =&Q(G_\alpha u, v)+\alpha \langle G_\alpha u,v\rangle_\ltwo\\
  =&\langle u,v\rangle_\ltwo=\langle u,v\rangle_{L^2(A,m)}.
  \end{align*}

It is left to show that $Q_A$ is densely defined in $L^2(A,m)$ if $Q$ is densely defined in $L^2(X,m)$.  
For every $f\in L^2(A,m)$, there is a sequence of functions $(f_n)$ in $D(Q)$ such that 
 $\|f_n-f\|_\ltwo\to 0$. By $Q$-invariance, we have $ f_n1_A\in D(Q) $ and, thus, by definition of $Q_A$, we infer $f_n1_A\in D(Q_A)$ for every $n$. Moreover, $f_n1_A$ converges to $f$ in $L^2(A,m)$. 
 Thus, density of $D(Q_A)$ in $L^2(A,m)$ follows. 
\end{proof}

Next we relate $Q_A$-invariance and $Q$-invariance  for restrictions $Q_A$ of $Q$.
\begin{lemma}\label{InvTrans}
Let $Q$ be a closed form in the wide sense. 
\begin{itemize}
 \item [(a)]  Every $ Q_{A} $ invariant subset  of a  $Q$-invariant set $ A $ is $Q$-invariant.
\item[(b)] The intersection $B\cap A$ of a measurable set $A$ and a $Q$-invariant set $ B $  is $Q_A$-invariant. 
\end{itemize}
 \end{lemma}
\begin{proof}
(a) Let $g\in L^2(X,m)$. By Lemma~\ref{RestrRes} and $Q$-invariance of $A$, the resolvent $G_\alpha^A$ of $Q_A$ is given by $G_\alpha|_{L^2(A,m)}$, $\alpha>0$. Hence, 
$$G_\alpha 1_Bg=G_\alpha|_{L^2(A,m)}1_Bg=G_\alpha^A 1_Bg.$$
Using the $Q_A$-invariance of $B$ we deduce $$G_\alpha^A 1_Bg=1_BG_\alpha^Ag.$$ By the equality $G_\alpha^Ag=1_AG_\alpha g$ we conclude the proof (a).

(b) Let $g\in D(Q_A)$. Snce $g1_{X\setminus A}=0$, we infer $g1_{A\cap B}=g1_{B}\in D(Q)$ by $Q$-invariance of $B$. 
By $A\cap B\subseteq A$, the equality
$g1_{A\cap B}1_{X\setminus A}= 0$ holds as well. Thus, we deduce $g1_{A\cap B}\in D(Q_A)$ and, therefore, $g1_{A\setminus(A\cap B)}=g1_{A\setminus B}=g1_{X\setminus B} \in D(Q_A)$, where the last equality holds by $g1_{X\setminus A}=0$. 
Furthermore, \[Q_A(g1_{A\cap B},g1_{A\setminus(A\cap B)})=Q(g1_B, g1_{X\setminus B})=0\] holds by $Q$-invariance of $B$. This concludes the proof.
\end{proof}

\bibliography{Literature}{}
\bibliographystyle{alpha}
\end{document}